\documentclass{article}
\usepackage{amsfonts,latexsym,hyperref} 
\usepackage{amsmath}
\newcounter{satznum}
\newtheorem{theorem}{Theorem}[satznum]

\newtheorem{corollary}[theorem]{Corollary}
\newtheorem{proposition}[theorem]{Proposition}

\newtheorem{example}[theorem]{Example}

\newtheorem{lemma}[theorem]{Lemma}

\newenvironment{remark}
{\begin{trivlist}\item[]{\bf Remark.}}
	{\end{trivlist}}

\newenvironment{proof}
{\begin{trivlist}\item[]{\bf Proof.}}
	{\end{trivlist}}
{\makeatletter
	\gdef\cz{{\mathbb C}} 
	\gdef\me{{\mathbb E}} 
	\gdef\nz{{\mathbb N}} 
	\gdef\pr{{\mathbb P}} 
	\gdef\rz{{\mathbb R}} 
}
\newcounter{todocounter}

\makeatletter
\def\@MRExtract#1 #2!{#1}
\newcommand{\MR}[1]{
	\xdef\@MRSTRIP{\@MRExtract#1 !}
	\href{http://www.ams.org/mathscinet-getitem?mr=\@MRSTRIP}{MR\@MRSTRIP}}
\makeatother

\begin{document}
\section*{Scaling limits for the block counting process and the fixation line of a class of $\Lambda$-coalescents}
{\sc Martin M\"ohle} and {\sc Benedict Vetter}\footnote{Mathematisches
	Institut, Eberhard Karls Universit\"at T\"ubingen, Auf der Morgenstelle 10,
	72076 T\"ubingen, Germany, E-mail addresses: martin.moehle@uni-tuebingen.de,
	benedict.vetter@uni-tuebingen.de}
\begin{center}
	\today
\end{center}
\begin{abstract}
	
	\vspace{2mm}
	
	\noindent We provide scaling limits for the block counting process and the fixation line of $\Lambda$-coalescents as the initial state $n$ tends to infinity under the assumption that the measure $\Lambda$ on $[0,1]$ satisfies $\int_{[0,1]} u^{-1}(\Lambda-b\lambda)({\rm d}u)<\infty$ for some $b>0$. Here $\lambda$ denotes the Lebesgue measure. The main result states that the block counting process, properly logarithmically scaled, converges in the Skorohod space to an Ornstein--Uhlenbeck type process as $n$ tends to infinity. The result is applied to beta coalescents with parameters $1$ and $b>0$. 
	We split the generators into two parts by additively decomposing $\Lambda$ and then prove the uniform convergence of both parts separately.
	
	\noindent Keywords: coalescent process; block counting process; fixation line; Bolthausen--Sznitman coalescent; Ornstein--Uhlenbeck type process; self-decomposability; generalized Mehler semigroup; time-inhomogeneous process; weak convergence.
	
	\vspace{2mm}
	
	\noindent 2020 Mathematics Subject Classification:
	Primary
	60J90; 
	Secondary
	60J27.  
\end{abstract}

\subsection{Introduction} \label{intro} 

The $\Lambda$-coalescent, independently introduced by Pitman \cite{pitman99} and Sagitov \cite{sagitov99}, is a Markov process $\Pi=(\Pi_t)_{t\geq 0}$ with c\`{a}dl\`{a}g paths, values in the space of partitions of $\nz:=\{1,2,\ldots\}$, starting at time $t=0$ from the partition $\{\{1\},\{2\},\ldots\}$ of $\nz$ into singletons, whose behavior is fully determined by a finite measure $\Lambda$ on the Borel subsets of $[0,1]$. If the process is in a state with $k\geq 2$ blocks, any particular $j\in\{2,\ldots,k\}$ blocks merge at the rate
\begin{eqnarray*}
	\lambda_{k,j}=\int_{[0,1]} u^{j-2}(1-u)^{k-j}\Lambda({\rm d}u).
\end{eqnarray*} 
The reader is referred to \cite{berestycki09} for a survey of $\Lambda$-coalescents. Unless $\Lambda(\{1\})>0$, either $\Pi_t$ has infinitely many blocks for all $t>0$ almost surely or finitely many blocks for all $t>0$ almost surely. The $\Lambda$-coalescent is said to stay infinite in the first case and to come down from infinity in the second. An atom of $\Lambda$ at $1$ corresponds to the rate of jumping to the trivial (and absorbing) partition consisting only of the block $\nz$. For $t\geq0$ let $N_t^{(n)}$ denote the number of blocks of the restriction $\Pi_t^{(n)}:=\{B\cap[n]|B\in\Pi_t, B\cap [n]\neq\emptyset\}$ of $\Pi_t$ to $[n]:=\{1,\ldots,n\}$. The \textit{block counting process} $N^{(n)}:=(N_t^{(n)})_{t\geq0}$ is a $[n]$-valued conservative Markov process with c\`{a}dl\`{a}g paths that jumps from $k\geq 2$ to $j\in\{1,\ldots,k-1\}$ at the rate
\begin{eqnarray*}
	q_{k,j}
	=\binom{k}{j-1}\int_{[0,1]} u^{k-j-1}(1-u)^{j-1}\Lambda({\rm d}u).
\end{eqnarray*}
Clearly, $N^{(n)}$ starts in $n$ at time $t=0$, has decreasing paths and eventually reaches the absorbing state $1$. Our main objective is to analyze the limiting behavior of the block counting process of $\Lambda$-coalescents that stay infinite as the initial state $n$ tends to infinity by determining suitable scaling constants. For $\Lambda$-coalescents with dust the scaling is $n$. A block $B \in \Pi_t$ of size $|B|=1$ is called a singleton. The number of singletons in $[n]$ divided by $n$ converges to the frequency of singletons as $n$ tends to infinity, and, if $\int_{[0,1]} u^{-1}\Lambda({\rm d}u)<\infty$ and $\Lambda(\{0\})=0$, then the frequency of singletons is strictly positive or the $\Lambda$-coalescent is said to have dust. It is known that $N^{(n)}/n$ converges in the Skorohod space $D_{[0,1]}[0,\infty)$ to the frequency of singletons process as $n$ tends to infinity \cite[Theorem 2.13]{gaisermoehle16}. 

The Bolthausen--Sznitman coalescent, where $\Lambda$ is the uniform distribution on $[0,1]$, is an example of a dust-free $\Lambda$-coalescent that stays infinite \cite{bolthausensznitman98}. In this case it was shown in \cite{moehle15} via moment calculations that $(N_t^{(n)}/n^{\exp(-t)})_{t\geq0}$ converges in $D_{[0,\infty)}[0,\infty)$ to the Mittag--Leffler process as $n$ tends to infinity. 

We provide unified proofs for both limit theorems, stated later as Corollaries \ref{thm_Dust} and \ref{thm_BS}, and extend the convergence to $\Lambda$-coalescents, where the measure $\Lambda$ is the sum of the uniform distribution on $[0,1]$ multiplied by a constant $b\geq0$ and a measure that corresponds to a coalescent with dust. This includes the $\Lambda$-coalescent where $\Lambda=\beta(1,b)$ is the beta distribution with parameters $1$ and $b>0$. Theorem \ref{thm_main_convergence} below states that $(\log N_t^{(n)}-\exp(-bt)\log n)_{t\geq0}$ converges in $D_{\rz}[0,\infty)$ as $n$ tends to infinity. The logarithmic version of the convergence result has the advantage of putting the limiting process in Theorem \ref{thm_main_convergence} to a class of processes, which has been studied in the literature. The limiting process can be represented as the solution of the Langevin equation with L\'{e}vy noise instead of a Brownian motion and is sometimes called Ornstein--Uhlenbeck type process \cite{satoyamazato84} or generalized Ornstein--Uhlenbeck process. First Wolfe \cite{wolfe82} studied Ornstein--Uhlenbeck type processes on $\rz$, later Jurek and Vervaat \cite{jurekvervaat83} on Banach spaces and Sato and Yamazato \cite{satoyamazato84} on $\rz^d$, then Applebaum \cite{applebaum07} on Hilbert spaces. Under the logarithmic moment condition (\ref{eq_lim_integral_cond}) the marginal distributions of the Ornstein--Uhlenbeck type process converge in distribution as time tends to infinity to the unique stationary distribution. The stationary distribution is self-decomposable. A real-valued random variable $S$ is \textit{self-decomposable} if for every $\alpha\in[0,1]$ there exists a random variable $S_{\alpha}$ independent of $S$ such that $S$ has the same distribution as $\alpha S+S_{\alpha}$. If $\phi$ is the characteristic function of $S$ then $S$ is self-decomposable if and only if $x\mapsto\phi(x)/\phi(\alpha x),$ $x\in\rz,$ is the characteristic function of a real-valued random variable for every $\alpha\in[0,1]$. A distribution $\mu$ on $\rz$ is self-decomposable if there exists a self-decomposable random variable with distribution $\mu$. Conversely, every self-decomposable distribution can be obtained in this way. For the $\beta(1,b)$-coalescent (\ref{eq_lim_integral_cond}) is satisfied. 

The fixation line $L=(L_t)_{t\geq0}$ is a $\nz$-valued Markov process that jumps from $k\in\nz$ to $j\in\{k+1,\ldots\}$ at the rate
\begin{eqnarray*}
	\gamma_{k,j}=\binom{j}{j-k+1}\int_{[0,1]}u^{j-k-1}(1-u)^k\Lambda({\rm d}u). 
\end{eqnarray*}
The fixation line is the 'time-reversal' of the block counting process, in the sense that the hitting times $\inf\{t\geq0|N_t^{(n)}\leq m\}$ and $\inf\{t\geq0|L_t^{(m)}\geq n\}$ share the same distribution \cite[Lemma 2.1]{henard15}. Equivalently, the process $L$ is Siegmund-dual \cite{siegmund76} to the block counting process, i.e., if $L^{(m)}=(L_t^{(m)})_{t\geq0}$ denotes the fixation line starting in $L_0^{(m)}=m,$ $m\in\nz,$ at time $t=0$ then (see \cite{KuklaMoehle18})
\begin{eqnarray}
	\pr(L_t^{(m)}\geq n)=\pr(N_t^{(n)}\leq m),\qquad m,n\in\nz,t\geq 0.
\label{eq_res_dual_discrete}
\end{eqnarray}
For a thorough definition of the fixation line see \cite{henard15} and the references therein. Theorem \ref{thm_res_fix_line} states that $(\log L_t^{(n)}-\exp(bt)\log n)_{t\geq0}$ converges in $D_{\rz}[0,\infty)$ as the initial value $L_0^{(n)}=n$ tends to infinity.

The article is organized as follows. In Section \ref{sec_results} the main convergence result (Theorem \ref{thm_main_convergence}) is stated and applied to the $\beta(1,b)$-coalescent with parameter $b>0$. The limiting process is an Ornstein--Uhlenbeck type process. Well-known results are applied to our setting in Section \ref{sec_limiting_process}. In particular, the generator of the limiting process is determined. The line of proof is in some sense reversed. First we prove Corollaries \ref{thm_Dust} and \ref{thm_BS} in Sections \ref{sec_proof_dust} and \ref{sec_proof_BS} by showing the convergence of the generators of the (logarithm of the) scaled block counting processes. The decomposition of $\Lambda$ into the uniform distribution multiplied by a constant and a measure that corresponds to a coalescent with dust is transferred to the generators. This enables us to use relations obtained in Sections \ref{sec_proof_dust} and \ref{sec_proof_BS} to prove Theorem \ref{thm_main_convergence} in Section \ref{sec_main}. The proof of Theorem \ref{thm_res_fix_line} is conducted in Section \ref{sec_proof_fix_line} with similar methods.

\noindent \textbf{Notation.} Let $E$ be a complete separable metric space. The Banach space $B(E)$ of bounded measurable functions $f:E\to\rz$ is equipped with the usual supremum norm $\|f\|:=\sup_{x\in E}|f(x)|$ and the Banach subspace $\widehat{C}(E)\subset B(E)$ consists of continuous functions vanishing at infinity. If $E\subseteq\rz^d$ for some $d\in\nz$ then $C_k(E)$ denotes the space of $k$-times continuously differentiable functions. The Borel-$\sigma$-field on $\rz$ is denoted by $\mathcal{B}$, $\Lambda$ is a (non-zero) finite measure on $\mathcal{B}\cap[0,1]$ with $\Lambda(\{0\})=\Lambda(\{1\})=0$ and $\lambda$ denotes the Lebesgue measure on $([0,1],\mathcal{B}\cap[0,1])$. The generators, usually denoted by $A$, are understood to be defined on (a subspace of) $\widehat{C}(E)$. For a measure space $(\Omega,\mathcal{F},\mu)$ and $p>0$ the space of measurable functions $f:\Omega\to\rz$ with $\int |f|^p {\rm d}\mu<\infty$ is denoted by $L^{p}(\mu)$ or in short, $L^{p}$. 

\subsection{Results} \label{sec_results}

Let $\Lambda$ be a finite measure on $([0,1],\mathcal{B}\cap[0,1])$ with no mass at $0$ and $1$. For $b\geq0$ the map $B\mapsto(\Lambda-b\lambda)(B)=\Lambda(B)-b\lambda(B),$ $B\in\mathcal{B}\cap[0,1],$ might possibly be a signed measure. Hahn's decomposition theorem states the existence of some $A\in\mathcal{B}\cap[0,1]$ such that $(\Lambda-b\lambda)^{+}(B):=(\Lambda-b\lambda)(B\cap A),$ $B\in\mathcal{B}\cap[0,1],$ and $(\Lambda-b\lambda)^{-}(B):=-(\Lambda-b\lambda)(B\cap A^c),$ $B\in\mathcal{B}\cap[0,1],$ define nonnegative measures. The two nonnegative measures $(\Lambda-b\lambda)^{+}$ and $(\Lambda-b\lambda)^{-}$ constitute the Jordan decomposition of $\Lambda-b\lambda$. Using this decomposition one can integrate with respect to a signed measure by defining $\int f{\rm d}(\Lambda-b\lambda):=\int f{\rm d}(\Lambda-b\lambda)^{+}-\int f{\rm d}(\Lambda-b\lambda)^{-}$ for $f\in L^{1}((\Lambda-b\lambda)^{+})\cap L^{1}((\Lambda-b\lambda)^{-})$. The assumption of Theorem \ref{thm_main_convergence} below is the following.

\vspace{2mm}

\noindent\textbf{Assumption A.} There exists $b\geq0$ such that $\int_{[0,1]} u^{-1}(\Lambda-b\lambda)^{+}({\rm d}u)<\infty$ and $\int_{[0,1]} u^{-1}(\Lambda-b\lambda)^{-}({\rm d}u)<\infty$.

\vspace{2mm}

\noindent Note that the constant $b\geq0$ is uniquely determined by $\Lambda$, if it exists. Schweinsberg's criterion \cite{schweinsberg00} shows that the $\Lambda$-coalescent does not come down from infinity under Assumption A, see Lemma \ref{lem_app_not_cdi} in the appendix. Moreover, the $\Lambda$-coalescent is dust-free if $b>0$. Assumption A is for example satisfied, if $\Lambda$ has density $f\in C_1([0,1])$ with respect to $\lambda$ for which $\lim_{u\searrow 0}f'(u)$ exists and is finite. In this case $b=\lim_{u\searrow0}f(u)$. 

Suppose that $\Lambda$ satisfies Assumption A. Let $\Gamma(z):=\int_{0}^{\infty}u^{z-1}e^{-u}{\rm d}u,$ ${\rm Re}(z)>0,$ denote the gamma function and $\Psi(z):=(\log\Gamma)^{\prime}(z)=\Gamma'(z)/\Gamma(z),$ ${\rm Re}(z)>0,$ the digamma function. Define 
\begin{eqnarray}
	a:=b(1+\Psi(1))-\int_{[0,1]} u^{-1}(\Lambda-b\lambda)({\rm d}u)
\label{eq_results_a}
\end{eqnarray}
and $\psi:\rz\to\cz$ via
\begin{eqnarray}
	\psi(x):=iax+\int_{[0,1]}(e^{ix\log(1-u)}-1+ixu)u^{-2}\Lambda({\rm d}u),\qquad x\in\rz.
\label{eq_results_Psi}
\end{eqnarray}
Substituting $g:(0,1)\to\rz,$ $g(u):=\log(1-u),$ $u\in(0,1)$, 
shows that
\begin{eqnarray*}
	\psi(x)=iax+\int_{(-\infty,0)}(e^{ixu}-1+ix(1-e^{u}))\varrho({\rm d}u),\qquad x\in\rz,
\end{eqnarray*}
where the measure $\varrho$, defined via 
\begin{eqnarray}
	\varrho(A):=\int_{g^{-1}(A)}u^{-2}\Lambda({\rm d}u)=\int_A(1-e^{u})^{-2}\Lambda_g({\rm d}u),\qquad A\in\mathcal{B},
\label{eq_res_varrho}
\end{eqnarray} 
satisfies $\int_{\rz}(u^2\wedge 1)\varrho({\rm d}u)<\infty$ and $\varrho(\{0\})=0$. Hence $\varrho$ is a L\'{e}vy measure and $e^{\psi(x)},$ $x\in\rz,$ is the characteristic function of an infinitely divisible distribution. 

\begin{theorem}
	Suppose that $\Lambda$ satisfies Assumption A. Then the possibly time-inhomogeneous 
	Markov process $X^{(n)}:=(X_t^{(n)})_{t\geq 0}:=(\log N_t^{(n)}-e^{-bt}\log n)_{t\geq 0}$ converges in $D_{\rz}[0,\infty)$ as $n\to\infty$ to the time-homogeneous Markov process $X=(X_t)_{t\geq 0}$ with initial value $X_0=0$ and semigroup $(T_t)_{t\geq0}$ given by
	\begin{eqnarray}
		T_tf(x):=\me(f(X_{s+t})|X_s=x)=\me(f(e^{-bt}x+X_{t})),\qquad x\in\rz,f\in B(\rz),s,t\geq 0,
	\label{eq_results_semigroup}
	\end{eqnarray}
	where $X_{t}$ has characteristic function $\phi_t$ given by
	\begin{eqnarray}
		\phi_t(x)=\exp\Big(\int_{0}^{t}\psi(e^{-bs}x){\rm d}s\Big),
	\qquad x\in\rz,t\geq 0,
	\label{eq_results_char_funct_X_t}
	\end{eqnarray}
	and $\psi$ is given by (\ref{eq_results_Psi}).
	\label{thm_main_convergence}
\end{theorem}

\noindent Distributions with characteristic functions of the form (\ref{eq_results_char_funct_X_t}), where $\psi$ is more generally the characteristic exponent of an infinitely divisible distribution, have been studied in \cite{wolfe82} on $\rz$ and on more general state spaces in \cite{applebaum07}, \cite{jurekvervaat83} and \cite{satoyamazato84}. Since $\phi_{t+s}(x)=\phi_t(e^{-bs}x)\phi_s(x),$ $x\in\rz,$ for $s,t\geq0$, the semigroup $(T_t)_{t\geq0}$ belongs to the class of generalized Mehler semigroups \cite{bogachevroecknerschmuland96}. 
In particular, the semigroup $(T_t)_{t\geq0}$ is Feller, i.e., $T_t(\widehat{C}(\rz))\subseteq\widehat{C}(\rz)$ for each $t\geq0$ and $(T_t)_{t\geq0}$ is strongly continuous on $\widehat{C}(\rz)$. Hence the Markov process $X$ in Theorem \ref{thm_main_convergence} exists and has paths in $D_{\rz}[0,\infty)$.

The beta distribution $\beta(a,b)$ with parameters $a,b>0$ has density $u\mapsto\Gamma(a+b)/(\Gamma(a)\Gamma(b))\linebreak u^{a-1}(1-u)^{b-1},$ $u\in(0,1),$ with respect to Lebesgue measure on $(0,1)$. The class of beta coalescents are the processes for which $\Lambda=\beta(a,b)$ for some $a,b>0$. They have been extensively studied in the literature due to the easy computability of the jump rates
\begin{eqnarray}
	q_{k,j}=\frac{\Gamma(a+b)\Gamma(k+1)\Gamma(j-1+b)\Gamma(k-j-1+a)}{\Gamma(a)\Gamma(b)\Gamma(k-2+a+b)\Gamma(j)\Gamma(k-j+2)},\qquad j\in\{1,\ldots,k-1\},k\geq 2.
\label{eq_res_block_jump_rates_beta}
\end{eqnarray}
The beta coalescent comes down from infinity if and only if $0<a<1$ \cite[Example 15]{schweinsberg00}. If $a>1$ then the beta coalescent has dust. For $a=1$ the beta coalescent is dust-free and does not come down from infinity.

\begin{example}
	Suppose that $\Lambda=\beta(1,b)$ with $b>0$. From the observation stated below Assumption A we conclude that Assumption A is satisfied with the same constant $b$. The 'dust-part' $\Lambda-b\lambda$ has possibly negative density $u\mapsto b((1-u)^{b-1}-1),$ $u\in(0,1),$ with respect to Lebesgue measure on $(0,1)$. The underlying L\'evy measure $\varrho$ has density $f$ with respect to Lebesgue measure on $\rz\setminus\{0\}$ given by $f(u):=be^{bu}(1-e^u)^{-2}$ for $u<0$ and $f(u):=0$ for $u>0$. Calculations involving Gau\ss' representation \cite[p. 247]{whittakerwatson96} for the digamma function $\Psi$ (see Proposition \ref{prop_app_underlying_char_funct_BS} in the appendix) show that $a=b(1+\Psi(b))$ and 
	\begin{eqnarray}
		\psi(x)	=b((1-b)\Psi(b)-(1-b-ix)\Psi(b+ix)),\qquad x\in\rz.
	\label{eq_res_ex_local_1}	
	\end{eqnarray} 
	According to Theorem \ref{thm_main_convergence} the process $(\log N_t^{(n)}-e^{-bt}\log n)_{t\geq 0}$ converges in $D_{\rz}[0,\infty)$ as $n\to\infty$ to a Markov process $X=(X_t)_{t\geq0}$ with initial value $X_0=0$ and semigroup $(T_t)_{t\geq0}$ given by
	\begin{eqnarray*}
		T_tf(x):=\me(f(X_{s+t})|X_s=x)=\me(f(e^{-bt}x+X_t)),\qquad x\in\rz,f\in B(\rz),s,t\geq0,
	\end{eqnarray*}
	where $X_t$ has characteristic function $\phi_t$ given by (\ref{eq_results_char_funct_X_t}). Since $\int_{(1-e^{-1},1)}\log \log (1-u)^{-1}\Lambda({\rm d}u)=\int_{1-e^{-1}}^{1}\log \log (1-u)^{-1}b(1-u)^{b-1}{\rm d}u 
	<\infty$, the logarithmic moment condition of Lemma \ref{lem_lim_process_stat_dist} is satisfied and $X_t$ converges in distribution as $t\to\infty$ to the unique stationary distribution $\mu$ of $X$. The distribution $\mu$ is self-decomposable with characteristic function $\phi$ given by
	\begin{eqnarray}
		\phi(x)=\exp\Big(\int_{0}^{\infty}\psi(e^{-bs}x){\rm d}s\Big)=\exp\Big((1-b)\int_{0}^{x}\frac{\Psi(b)-\Psi(b+iu)}{u}{\rm d}u\Big)\frac{\Gamma(b+ix)}{\Gamma(b)},\quad x\in\rz.
	\label{eq_res_ex_local}
	\end{eqnarray}
	In the last step equation (\ref{eq_res_ex_local_1}) and the fact that $\Psi(z)=(\log \Gamma(z))',$ ${\rm Re}(z)>0,$ has been used. The characteristic function $\phi_t$ of $X_t$ is hence given by
	\begin{eqnarray*}
		\phi_t(x)
		=\frac{\phi(x)}{\phi(e^{-bt} x)}
		=\exp\Big((1-b)\int_{e^{-bt}x}^{x}\frac{\Psi(b)-\Psi(b+iu)}{u}{\rm d}u\Big)\frac{\Gamma(b+ix)}{\Gamma(b+ie^{-bt}x)},\quad x\in\rz,t\geq0.
	\end{eqnarray*} 
	If $Z$ has a gamma distribution with parameter $b$ and $1$, i.e., $Z$ has density $u\mapsto u^{b-1}e^{-u}(\Gamma(b))^{-1},$ $u>0$, with respect to Lebesgue measure on $(0,\infty)$ then $\log Z$ has a self-decomposable distribution and characteristic function $\Gamma(b+ix)/\Gamma(b),$ $x\in\rz,$ see \cite[V, Example 9.18]{steutelvanharn04}. If $b<1$ then the first factor on the right-hand side of (\ref{eq_res_ex_local}) is the characteristic function of a self-decomposable distribution (see \cite[V, Theorem 6.7]{steutelvanharn04} and the proof of Proposition \ref{prop_app_underlying_char_funct_BS}). The underlying characteristic exponent $(1-b)(\Psi(b)-\Psi(b+iu)),$ $u\in\rz,$ corresponds to the negative of a drift-free subordinator. Similarly to the convergence above, $(N_t^{(n)}/n^{e^{-bt}})_{t\geq0}$ converges in $D_{[0,\infty)}[0,\infty)$ to $(\exp(X_t))_{t\geq0}$ as $n\to\infty$.
\label{ex_res}
\end{example}

\noindent The two cases mentioned in the introduction arise from Assumption A as follows. If $\int_{[0,1]} u^{-1}\Lambda({\rm d}u)<\infty$, then the $\Lambda$-coalescent has dust and Assumption A is satisfied with $b=0$. Corollary \ref{thm_Dust} below has been proven in \cite{gaisermoehle16} and \cite{moehle21}. In both articles the blocks of the coalescent are allowed to merge simultaneously. In \cite{moehle21} the convergence of the generators has been proven and even a rate of convergence has been determined. In this article the uniform convergence of the generators is going to be proven as well, but with different techniques. In \cite{gaisermoehle16} the convergence of the corresponding semigroups has been shown, which is equivalent to the convergence of the generators on a core. We carry out the proof since parts are used to verify Theorem \ref{thm_main_convergence}. 

\begin{corollary}[dust case] Suppose $\int_{[0,1]} u^{-1}\Lambda({\rm d}u)<\infty$. Then the time-homogeneous Markov process 
	$X^{(n)}:=(X_t^{(n)})_{t\geq 0}:=(N_t^{(n)}-\log n)_{t\geq 0}$ converges in $D_{\rz}[0,\infty)$ as $n\to\infty$ to a limiting process $X=(X_t)_{t\geq 0}$ with initial value $X_0=0$ and semigroup $(T_t)_{t\geq0}$ given by 
	\begin{eqnarray}
		T_tf(x):=\me(f(X_{s+t})|X_s=x)=\me(f(x+X_t)),\qquad x\in\rz,f\in B(\rz),s,t\geq 0,
	\label{eq_res_semigroup_dust}
	\end{eqnarray}
	where $X_t$ has characteristic function $\me(\exp(ixX_t))=\exp(t\psi(x)),$ $x\in\rz,t\geq 0,$ with 
	\begin{eqnarray}
		\psi(x)=\int_{[0,1]} (e^{ix\log(1-u)}-1)u^{-2}\Lambda({\rm d}u),\qquad x\in\rz.
	\label{eq_res_char_exp_dust}
	\end{eqnarray}
	Observe that $-X$ is a pure-jump subordinator with characteristic exponent $x\mapsto\psi(-x),$ $x\in\rz$.
	\label{thm_Dust}	
\end{corollary}

\noindent Note that here $b=0$, $a=\int_{[0,1]}u^{-1}\Lambda({\rm d}u)$ and the definitions (\ref{eq_results_Psi}) and (\ref{eq_res_char_exp_dust}) for $\psi$ coincide. Hence Theorem \ref{thm_main_convergence} and Corollary \ref{thm_Dust} describe the same limiting process. 

For $\Lambda=\lambda$ Assumption A is satisfied with $b=1$. The block counting process of the Bolthausen--Sznitman coalescent has been treated in \cite{KuklaMoehle18} and \cite{moehle15}. \cite{KuklaMoehle18} and \cite{moehle15} have proven that the semigroup of $(N_t^{(n)}/n^{e^{-t}})_{t\geq0}$ converges on a dense subset of $B([0,\infty))$ to the semigroup of a Feller process as $n$ tends to infinity, hence the processes converge in $D_{[0,\infty)}[0,\infty)$. Taking logarithms does not spoil the convergence. If $f\in\widehat{C}(\rz)$ then $g:=f\circ\log\in\widehat{C}([0,\infty))$, and the semigroup and hence the generator $A^{(n)}$ of the logarithm of the scaled block counting process $X^{(n)}=(X_t^{(n)})_{t\geq0}=(\log N_t^{(n)}-e^{-t}\log n)_{t\geq0}$ converge as well. We prove the convergence of $A^{(n)}$ in Section \ref{sec_proof_BS} directly. Since the scaling depends on $t$, the process $X^{(n)}$ is time-inhomogeneous, and  \cite{KuklaMoehle18} introduces the time-space process in order to transfer the question of convergence to time-homogeneous Markov processes. The time-space process is revisited in Section \ref{sec_proof_BS}. Since $\lambda=\beta(1,1)$, the following result is the particular case $b=1$ of Example \ref{ex_res}.

\begin{corollary}[Bolthausen--Sznitman case]
	Suppose $\Lambda=\lambda$. Then the time-inhomogeneous Markov process $X^{(n)}:=(X_t^{(n)})_{t\geq 0}:=(N_t^{(n)}-e^{-t}\log n)_{t\geq 0}$ converges in $D_{\rz}[0,\infty)$ as $n\to\infty$ to the time-homogeneous Markov process $X=(X_t)_{t\geq 0}$ with initial value $X_0=0$ and semigroup $(T_t)_{t\geq0}$ given by
	\begin{eqnarray*}
		T_tf(x):=\me(f(X_{s+t})|X_s=x)=\me(f(e^{-t}x+X_t)),\qquad x\in\rz,f\in B(\rz),s,t\geq 0,
	\end{eqnarray*}
	where $X_t$ has characteristic function $\phi_t(x):=\me(\exp(ixX_t))=\Gamma(1+ix)/\Gamma(1+ie^{-t}x),$ $x\in\rz,t\geq 0$.
	\label{thm_BS}
\end{corollary}

\noindent Note that here $b=1,$ $a=1+\Psi(1)$ and the underlying L\'{e}vy measure $\varrho$ has density $f$ with respect to the Lebesgue measure on $\rz\setminus\{0\}$ given by $f(u):=e^{u}(1-e^{u})^{-2}$ for $u<0$ and $f(u):=0$ for $u>0$. Example \ref{ex_res} with $b=1$ states that $\psi(x)=ix\Psi(1+ix),$ $x\in\rz,$ and that $X_t$ converges in distribution as $t\to\infty$ to the unique stationary distribution $\mu$ of $X$ with characteristic function $\phi(x)=\Gamma(1+ix),$ $x\in\rz$. Let $Z$ have an exponential distribution with parameter $1$. Then (see e.g. \cite[V, Example 9.15]{steutelvanharn04}) $\log Z$ is the negative of a Gumbel distributed random variable and $\me(e^{ix\log Z})=\Gamma(1+ix),$ $x\in\rz$. Hence $-X_t$ converges in distribution as $t\to\infty$ to the Gumbel distribution.

A convergence result for the fixation line can be stated analogously to Theorem \ref{thm_main_convergence}.

\begin{theorem}
	Suppose that $\Lambda$ satisfies Assumption A. Then the possibly time-inhomogeneous Markov process $Y^{(n)}:=(Y_t^{(n)})_{t\geq0}:=(\log L_t^{(n)}-e^{bt}\log n)_{t\geq0}$ converges in $D_{\rz}[0,\infty)$ as $n\to\infty$ to the time-homogeneous Markov process $Y=(Y_t)_{t\geq0}$ with initial value $Y_0=0$ and semigroup $(T_t)_{t\geq0}$ given by
	\begin{eqnarray}
		T_tf(y):=\me(f(Y_{s+t})|Y_s=y)=\me(f(e^{bt}y+Y_t)), \qquad y\in\rz,f\in B(\rz),s,t\geq0,
	\label{eq_res_semigroup_lim_fix_line}
	\end{eqnarray}
	where $Y_t$ has characteristic function $\chi_t$ given by
	\begin{eqnarray}
		\chi_t(y)=\exp\Big(\int_{0}^{t}\psi(-e^{bs}y){\rm d}s\Big),\qquad y\in\rz,t\geq 0,
	\label{eq_res_char_funct_Y_t}
	\end{eqnarray}
	and $\psi$ is given by (\ref{eq_results_Psi}).
\label{thm_res_fix_line}	
\end{theorem}

\begin{remark}
	\begin{enumerate}
	\item The process defined by (\ref{eq_res_semigroup_lim_fix_line}) and (\ref{eq_res_char_funct_Y_t}) is an Ornstein--Uhlenbeck type process with underlying characteristic exponent $y\mapsto \psi(-y),$ $y\in\rz.$ The semigroup defined by (\ref{eq_res_semigroup_lim_fix_line}) belongs to the class of generalized Mehler semigroups, since $\chi_{t+s}(y)=\chi_t(e^{bs}y)\chi_s(y),$ $y\in\rz,$ for $s,t\geq0$ \cite{bogachevroecknerschmuland96}.
	
	\item Let the random variable $S_t$ have characteristic function $\phi_t$, given by (\ref{eq_results_char_funct_X_t}), for $t\geq0$. Conditional on $X_s=x$, $X_{t+s}$ is distributed as $e^{-bt}x+S_t$ for all $x\in\rz$. Note that $Y_t\overset{\text{d}}{=}-e^{bt}X_t\overset{\text{d}}{=}-e^{bt}S_t$ and that conditional on $Y_s=y$, $Y_{t+s}$ is distributed as $e^{bt}y-e^{bt}S_t$ such that
	\begin{eqnarray*}
		\pr(e^{Y_{t+s}}\geq x|e^{Y_s}=y)=\pr(y^{e^{bt}}e^{-e^{bt}S_t}\geq x)=\pr(x^{e^{-bt}}e^{S_t}\leq y)=\pr(e^{X_{t+s}}\leq y|e^{X_s}=x)
	\end{eqnarray*}
	for all $x,y,s,t\geq 0$, i.e., $e^Y$ is Siegmund-dual to $e^X$ (see \cite{siegmund76}).
	\item The Bolthausen--Sznitman case $\Lambda=\lambda$ is stated in \cite[Theorem 3.1 b)]{KuklaMoehle18} in non-logarithmic form. The fixation line in the Bolthausen--Sznitman coalescent is a continuous-time discrete state space branching process in which the offspring distribution has probability generating function $f(s)=s+(1-s)\log(1-s),$ $s\in[0,1].$ 
	The limiting process described in Theorem \ref{thm_res_fix_line} is the logarithm of Neveu's continuous-state branching process. By Corollary \ref{thm_BS}, the characteristic functions $\chi_t$ of the marginal distributions are given by (see \cite[Eq. (19)]{KuklaMoehle18})
	\begin{eqnarray*}
		\chi_t(y)=\phi_t(-e^{t}y)=\Gamma(1-ie^{bt}y)/\Gamma(1-iy), \qquad y\in\rz,t\geq0.
	\end{eqnarray*}
	\end{enumerate}
\end{remark}

\subsection{The limiting process} \label{sec_limiting_process}

Standard computations (see \cite[Lemma 17.1]{sato99}) show that $\phi_t$, given by (\ref{eq_results_char_funct_X_t}), is the characteristic function of an infinitely divisible distribution for each $t\geq0$ without Gaussian component and L\'{e}vy measure $\varrho_t$ given by
\begin{eqnarray*}
	\varrho_t(A)=\int_{\rz\setminus\{0\}}\int_{0}^{t}1_A(e^{-bs}u){\rm d}s\varrho({\rm d}u),\qquad A\in\mathcal{B},t\geq0.
\end{eqnarray*}
Sato and Yamazato \cite[Theorem 3.1]{satoyamazato84} provide a formula for the generator corresponding to the semigroup $(T_t)_{t\geq0}$ given by (\ref{eq_results_semigroup}). 

\begin{lemma}
	Suppose that $\Lambda$ satisfies Assumption A. (Let $\psi$ be given by (\ref{eq_results_Psi}), $\phi_t$ be defined by (\ref{eq_results_char_funct_X_t}) and let the random variable $X_t$ have characteristic function $\phi_t$ for each for $t\geq0$.) The family of operators $(T_t)_{t\geq 0}$ defined by (\ref{eq_results_semigroup}) is a Feller semigroup. Let $D$ denote the space of twice differentiable functions $f:\rz\to\rz$ such that $f,f',f''\in\widehat{C}(\rz)$ and such that the map $x\mapsto xf'(x),$ $x\in\rz,$ belongs to $\widehat{C}(\rz)$. Then $D$ is a core for the generator $A$ corresponding to $(T_t)_{t\geq0}$ and
	\begin{eqnarray}
		Af(x)=f'(x)(a-bx)+\int_{[0,1]}(f(x+\log(1-u))-f(x)+uf'(x))u^{-2}\Lambda({\rm d}u)
	\label{eq_lim_process_generator}
	\end{eqnarray}
	for $x\in\rz$ and $f\in D$, where $a$ is given by (\ref{eq_results_a}).
\label{lem_lim_process_generator}
\end{lemma}

\begin{proof}
	Substituting $g:(0,1)\to\rz,$ $g(u):=\log(1-u),$ $u\in(0,1)$ shows that (\ref{eq_lim_process_generator}) is an integro-differential operator of the form (1.1) of Sato and Yamazato \cite{satoyamazato84} with dimension $d=1$. In \cite{satoyamazato84}, operators of this form are initially considered as acting on the space $C_c^2$ of twice differentiable functions with compact support (see the explanations after Eq.~(1.2) in \cite{satoyamazato84}), but Step 3 of the proof of \cite[Theorem 3.1]{satoyamazato84} shows that (\ref{eq_lim_process_generator}) even holds for functions $f\in D$ ($\supset C_c^2$). Note that the space $D$ is denoted by $F_1$ in \cite{satoyamazato84}. The fact that $D$ is a core for $A$ is only a different phrasing of the claim in Step 5 of the proof of \cite[Theorem 3.1]{satoyamazato84}.
\hfill$\Box$\end{proof}

\begin{remark} 
	The limiting process in Theorem \ref{thm_main_convergence} arises as the solution of a certain stochastic differential equation. Let the L\'{e}vy process $L=(L_t)_{t\geq0}$ with characteristic functions $\me(e^{ixL_t})=e^{t\psi(x)},$ $x\in\rz,t\geq0,$ be adapted to the filtration $(\mathcal{F}_t)_{t\geq0}$ which satisfies the usual hypotheses such that $L_{t+s}-L_{s}$ is independent of $\mathcal{F}_s$ for all $s,t\geq0$. In this remark $\psi$ is allowed to be the characteristic exponent of an arbitrary infinitely divisible distribution on $\rz$ and $b>0$ is fixed. The Langevin equation with L\'{e}vy noise instead of a Brownian motion 
	\begin{eqnarray}
		{\rm d}X_t=-bX_{t}{\rm d}t+{\rm d}L_t,\qquad t\geq 0,
	\label{eq_lim_langevin}
	\end{eqnarray}
	with initial value $X_0=0$ has an unique $(\mathcal{F}_t)_{t\geq0}$-adapted solution $X=(X_t)_{t\geq0}$ with c\`{a}dl\`ag paths. The solution of (\ref{eq_lim_langevin}) or the corresponding semigroup are hence sometimes called \textit{Ornstein--Uhlenbeck type} or \textit{generalized Ornstein--Uhlenbeck} process or semigroup. It holds that
	\begin{eqnarray}
		X_t= \int_{0}^{t}e^{-b(t-s)}{\rm d}L_s,\qquad t\geq 0.
	\label{eq_lim_stoch_int}
	\end{eqnarray}
	Various constructions for the integral (\ref{eq_lim_stoch_int}) are possible. In Applebaum \cite[Sections 6.3 and 6.2]{applebaum04} the stochastic integral is the It\^{o}-integral with respect to semimartingales. Wolfe \cite{wolfe82} constructed the integral as a random Bochner integral, which exists in the sense of convergence in probability. Jurek and Vervaat \cite{jurekvervaat83} constructed the stochastic integral as a pathwise Laplace-Stieltjes integral (and using integration by parts). 
	The process $X$ is a stochastically continuous Markov process and the corresponding semigroup is given by (\ref{eq_results_semigroup}), where the characteristic functions of $X_t$ are given by (\ref{eq_results_char_funct_X_t}) with underlying infinitely divisible characteristic exponent $\psi$ for $t\geq 0$.  

	Suppose that $b>0$ and that the L\'{e}vy measure $\varrho$ of the characteristic exponent $\psi$ satisfies
	\begin{eqnarray}
		\int_{\{|x|>1\}}\log(1+|u|)\varrho({\rm d}u)<\infty.
	\label{eq_lim_integral_cond}
	\end{eqnarray}
	According to \cite[Theorems 4.1 and 4.2]{satoyamazato84}, $X_t$ converges in distribution as $t\to\infty$ to the unique stationary distribution $\mu$ of $X$. The distribution $\mu$ is self-decomposable and conversely every self-decomposable distribution can be obtained as the stationary distribution of an Ornstein--Uhlenbeck type process. If (\ref{eq_lim_integral_cond}) does not hold, then there exists no stationary distribution. The following Lemma is an application to our setting.
\end{remark}

\begin{lemma}
	Suppose that $\Lambda$ satisfies Assumption A. Let $X=(X_t)_{t\geq0}$ be as in Theorem \ref{thm_main_convergence}. If further $\int_{(\varepsilon,1)} \log \log (1-u)^{-1}\Lambda({\rm d}u)<\infty$ for some $1-e^{-1}<\varepsilon<1$ then $X_t$ converges in distribution as $t\to\infty$ to the unique stationary distribution $\mu$ of $X$. The distribution $\mu$ is self-decomposable with characteristic function $\phi$ given by
	\begin{eqnarray}
		\phi(x)=\exp\Big(\int_{0}^{\infty}\psi(e^{-bs}x){\rm d}s\Big),\qquad x\in\rz.
	\end{eqnarray}
	The characteristic function $\phi_t$ of $X_t$ satisfies $\phi_t(x)=\phi(x)/\phi(e^{-bt}x),$ $x\in\rz$. 
	
	If $\int_{(\varepsilon,1)} \log \log (1-u)^{-1}(u^{-2})\Lambda({\rm d}u)=\infty$ for $0<\varepsilon<1$, then, for every $l$,
	\begin{eqnarray}
		\lim_{t\to\infty}\sup_{x\in\rz}\sup_{y\in\rz}\pr(|e^{-bt}x+X_t-y|\leq l)=0.
	\end{eqnarray}
	The process has no stationary distribution.
\label{lem_lim_process_stat_dist}
\end{lemma}

\subsection{Proof of Corollary \ref{thm_Dust}} \label{sec_proof_dust}

In this section $\Lambda$ satisfies the dust condition $\int_{[0,1]}u^{-1}\Lambda({\rm d}u)<\infty$. Let $E_n:=\{x\in\rz|e^{x}n\in[n]\}$ denote the state space of $X^{(n)}=(X_t^{(n)})_{t\geq 0}=(\log N_t^{(n)}-\log n)_{t\geq 0}$ for each $n\in\nz$. Define $k:=k(x,n):=e^{x}n\in [n]$ for $x\in E_n$ and $n\in\nz$ such that the generator $A^{(n)}$ of $X^{(n)}$ can be represented as
\begin{eqnarray}
	A^{(n)}f(x)=\sum_{j=1}^{k-1}(f(x+\log\tfrac{j}{k})-f(x))q_{k,j},\qquad x\in E_n,f\in\widehat{C}(\rz),n\in\nz.
\label{eq_proof_dust_gen_discrete}
\end{eqnarray}
The process $X=(X_t)_{t\geq 0}$ defined by (\ref{eq_res_semigroup_dust}) and (\ref{eq_res_char_exp_dust}) is a Feller process in $\widehat{C}(\rz)$. Let $A$ denote the generator. From \cite[Theorem 31.5]{sato99} it follows that the space $\widehat{C}_2(\rz)$ of twice differentiable functions $f\in C_2(\rz)$ with $f,f',f''\in\widehat{C}(\rz)$ is a core for $A$ and
\begin{eqnarray}
	Af(x)=\int_{[0,1]}(f(x+\log(1-u))-f(x))u^{-2}\Lambda({\rm d}u),\qquad x\in\rz,f\in \widehat{C}_2(\rz).
\label{eq_proof_dust_gen_limit}
\end{eqnarray}
The idea to prove the uniform convergence of the generators is the following: write the jump rates as values of a distribution depending on $k$ (with some minor rectifications) whose limiting behavior as $k\to\infty$ can be determined. The generator can then be written as the mean of a random variable and classical weak convergence results can be applied. 

\begin{proof} [of Corollary \ref{thm_Dust}]
	Let $f \in \widehat{C}_2(\rz)$. Define $h : [0,1] \times \rz \to \rz$ via $h(u,x):=u^{-1}(f(x+\log(1-u))-f(x)),$ $u\in(0,1),$ $h(0,x):=\lim_{u\searrow0}h(u,x)=-f'(x)$ and $h(1,x):=\lim_{u\nearrow1}h(u,x)=-f(x)$ for $x\in\rz$. Differentiating $s \mapsto f(x+\log(1-us)),$ $s\in(0,1),$ leads to
	\begin{eqnarray*}
		f(x+\log(1-u))-f(x)
		=-u\int_{0}^{1}\frac{f'(x+\log(1-us))}{1-us}{\rm d}s, \qquad u\in[0,1),x\in\rz,
	\end{eqnarray*}
	such that
	\begin{eqnarray*}
		h(u,x)=-\int_{0}^{1}\frac{f'(x+\log(1-us))}{1-us}{\rm d}s, \qquad u\in[0,1),x\in\rz,
	\end{eqnarray*}
	and $h$ stays bounded even as $u$ tends to $0$. Define
	\begin{eqnarray}
		S(k,x):=\sum_{j=1}^{k-1}(f(x+\log\tfrac{j}{k})-f(x))q_{k,j}, \quad I(x):=\int_{[0,1]}h(u,x)u^{-1}\Lambda({\rm d}u), \qquad k\in\nz,x\in\rz,
	\label{eq_proof_dust_sum_and_integral}
	\end{eqnarray} 
	such that $A^{(n)}f(x)=S(k,x)$ for $x\in E_n$ and $n\in\nz$ and $I(x)=Af(x)$ for $x\in\rz$. Substituting $k-j$ for $j$ and the definition of $h$ yield
	\begin{eqnarray*}
		S(k,x) 
		&=&\sum_{j=1}^{k-1}(f(x+\log(1-\tfrac{j}{k}))-f(x))q_{k,k-j}
		\\&=&\sum_{j=1}^{k-1}h(\tfrac{j}{k},x)\frac{j}{k}\binom{k}{j+1}\int_{[0,1]}u^{j-1}(1-u)^{k-j-1}\Lambda({\rm d}u)
		\\&=&\sum_{j=0}^{k-1}h(\tfrac{j}{k},x)\frac{j}{j+1}\binom{k-1}{j}\int_{[0,1]}u^{j-1}(1-u)^{k-j-1}\Lambda({\rm d}u).
	\end{eqnarray*}
	Set $c:=\int_{[0,1]} u^{-1}\Lambda({\rm d}u)>0$ and define the probability measure $\mathrm{Q}$ on $([0,1],\mathcal{B}\cap[0,1])$ via $\mathrm{Q}(A):=c^{-1}\int_Au^{-1}\Lambda({\rm d}u),$ $A\in\mathcal{B}\cap[0,1].$ Let the random variables $Z_k,$ $k\in\nz,$ have distribution given by
	\begin{eqnarray*}
		\pr(Z_k=j)=\binom{k-1}{j}\int_{[0,1]}u^j(1-u)^{k-1-j}\mathrm{Q}({\rm d}u),\qquad j\in\{0,\ldots,k-1\},
	\end{eqnarray*}
	i.e., $Z_k$ has a mixed binomial distribution with sample size $k-1$ and random success rate $\mathrm{Q}$. Let the random variable $Z$ have distribution $\mathrm{Q}$. Then
	\begin{eqnarray*}
		S(k,x)=c\me((1-(Z_k+1)^{-1})h(Z_k/k,x)),\qquad k\in\nz,x\in\rz, 
	\end{eqnarray*}
	and $I(x)=c\me(h(Z,x)),$ $x\in\rz.$ It is straightforward to check that $Z_k/k\to Z$ in distribution as $k\to\infty$, e.g., by verifying the convergence of the cumulative distribution functions (cdf) on the set of continuity points of the cdf of $Z$. In particular, $\lim_{k\to\infty}\pr(Z_k\leq C)=Q(0)=0$ for every $C>0$ such that $\lim_{k\to\infty}\me((Z_k+1)^{-1})=0$. Since $h$ is bounded and $f,f'\in\widehat{C}(\rz)$ are uniformly continuous, the family of functions $\{h(\cdot,x)|x\in\rz\}$ is equicontinuous on $[\delta,1-\delta]$ for every $0<\delta<1/2$ and uniformly bounded on $[0,1]$. From Lemma \ref{local_lemma} it follows that $\me(h(Z_k/k,x))\to\me(h(Z,x))$ uniformly in $x\in\rz$ as $k\to\infty$, thus
	\begin{eqnarray}
		\lim_{k\to\infty}\sup_{x\in\rz}|S(k,x)-I(x)|=0.
	\label{eq_proof_dust_1}
	\end{eqnarray}  
	From $\lim_{x\to-\infty}h(Z,x)=0$ a.s., the fact that $h$ is bounded and the dominated convergence theorem it follows that
	\begin{eqnarray}
		\lim_{x\to-\infty}|I(x)|=c\lim_{x\to-\infty}|\me(h(Z,x))|=0.
	\label{eq_proof_dust_2}
	\end{eqnarray}
	Since $f\in\widehat{C}(\rz)$, $\lim_{x \to -\infty}S(k,x)=0$ for any $k\in\nz$. Due to (\ref{eq_proof_dust_1}) and (\ref{eq_proof_dust_2}),
	\begin{eqnarray}
		\lim_{x\to-\infty}\sup_{k\in\nz}|S(k,x)|=0.
	\label{eq_proof_dust_3}
	\end{eqnarray}
	As $n\to\infty$, $k=k(x,n)=e^xn\to\infty$ or $x\to-\infty$. For example, for $n\in\nz$ and $x\in E_n$, either $k\geq n^{1/2}$ or $x<-\tfrac{1}{2}\log n$. Distinguishing the two cases leads to
	\begin{eqnarray}
		&&\lim_{n\to\infty}\sup_{x\in E_n}|A^{(n)}f(x)-Af(x)|\nonumber
		\\&&~~~~~~~~~~\leq\lim_{k\to\infty}\sup_{x\in\rz}|S(k,x)-I(x)|+\lim_{x\to-\infty}\sup_{k\in\nz}|S(k,x)|+\lim_{x\to-\infty}|I(x)|=0.
	\label{eq_proof_dust_gen_conv}	
	\end{eqnarray}
	By \cite[I, Theorem 6.1 and IV, Theorem 2.5]{ethierkurtz}, $X^{(n)}\to X$ in $D_{\rz}[0,\infty)$ as $n\to\infty$.
\hfill$\Box$\end{proof}

\begin{remark} 
	\mbox{}
	\begin{enumerate}
		\item The generator $A^{(n)}$ converges even if $\Lambda(\{1\})>0$. In this case the atom at $1$ can be split off from $\Lambda$ such that $q_{k,j}=\binom{k}{j-1}\int_{[0,1)}u^{k-j-1}(1-u)^{j-1}\Lambda|_{[0,1)}({\rm d}u)+\Lambda(\{1\})1_{\{1\}}(j),$ $j\in\{1,\ldots,k-1\},k\geq2,$ where the first summand are the jump rates of the block counting process corresponding to the restriction $\Lambda|_{[0,1)}$ of $\Lambda$ to $[0,1)$, i.e., a measure with no atom at $1$. Thus, 
		\begin{eqnarray*}
			A^{(n)}f(x)=S(k,x)+(f(\log n^{-1})-f(x))\Lambda(\{1\}),\qquad x\in E_n,f\in\widehat{C}(\rz),n\in\nz,
		\end{eqnarray*}
		where the jump rates in $S(k,x)$ correspond to $\Lambda|_{[0,1)}$, and 
		\begin{eqnarray*}
			Af(x)=I(x)+h(1,x)\Lambda(\{1\})=I(x)-f(x)\Lambda(\{1\}),\qquad x\in(-\infty,0],
		\end{eqnarray*}
		where $I(x)=\int h(u,x)\Lambda|_{[0,1)}({\rm d}u),$ $x\in\rz$. The additional term corresponds to the killing of the subordinator $-X$ at the rate $\Lambda(\{1\})$. Since $f\in\widehat{C}(\rz)$, $\lim_{n\to\infty}\sup_{x\in E_n}|(f(\log n^{-1})-f(x))\Lambda(\{1\})+f(x)\Lambda(\{1\})|=\Lambda(\{1\})\lim_{n\to\infty}|f(\log n^{-1})|=0$, i.e., the additional term converges, and again (\ref{eq_proof_dust_gen_conv}) holds true.
		\item The approach to the convergence of the generators is related to Bernstein polynomials. The $(k-1)$-th Bernstein polynomial 
		\begin{eqnarray*}
			\sum_{j=0}^{k-1}h(\tfrac{j}{k-1},x)\binom{k-1}{j}u^{j}(1-u)^{k-1-j}
		\end{eqnarray*} 
		of $h(\cdot,x)$ converges uniformly in $u\in[0,1]$ to $h(u,x)$ as $k\to\infty$, if $x\in\rz$ is fixed.
	\end{enumerate}
\end{remark}

\subsection{Proofs concerning the Bolthausen--Sznitman coalescent} \label{sec_proof_BS}

In this section $\Lambda = \lambda$ is the Lebesgue measure on $[0,1]$. Define $\alpha:=\alpha(t):=e^{-t},$ $t\geq 0$. 
The process $X^{(n)}=(X_t^{(n)})_{t\geq 0}=(\log N_t^{(n)}-\alpha\log n)_{t\geq 0}$ is a time-inhomogeneous Markov process. In order to prove convergence in $D_{\rz}[0,\infty)$ to $X$ we want to show the uniform convergence of the generators. Typical convergence results are stated for time-homogeneous Markov processes and in order to use these we are going to introduce the time-space process.

\subsubsection{Time-space process: semigroup and generator} \label{sec_proof_BS_time_space}

Define the time-space processes $\widetilde{X}^{(n)}:=(t,X_t^{(n)})_{t\geq0},$ $n\in\nz$, and $\widetilde{X}:=(t,X_t)_{t\geq0}$. It has been proven in \cite{boettcher13} that $\widetilde{X}^{(n)}$ and $\widetilde{X}$ are time-homogeneous Markov processes (and exist on a new probability space). In the following the tilde symbol indicates the time-space setting. Let $\widetilde{E}_n:=\{(s,x)\in[0,\infty)\times\rz|e^xn^{\alpha(s)}\in [n]\}$ denote the state space of $\widetilde{X}^{(n)}$, $\widetilde{E}:=[0,\infty)\times\rz$ denote the state space of $\widetilde{X}$ and define $k:=k(s,x,n):=e^xn^{\alpha(s)}\in\nz$ for $(s,x)\in\widetilde{E}_n$ and $n\in\nz$. Given $f\in B(\widetilde{E})$ and $s\geq0$, denote the function $x\mapsto f(s,x),$ $x\in\rz$, by $\pi f(s,x)$. The limiting process $X$ already is time-homogeneous. Recall that $D$, the space of twice differentiable functions $f:\rz\to\rz$ such that $f,f',f''\in\widehat{C}(\rz)$ and such that the map $x\mapsto xf'(x),$ $x\in\rz,$ belongs to $\widehat{C}(\rz)$, is a core for the generator $A$ of the semigroup $(T_t)_{t\geq0}$ corresponding to $X$. The semigroup $(\widetilde{T}_t)_{t\geq0}$ of $\widetilde{X}$, given by
\begin{eqnarray*}
	\widetilde{T}_tf(s,x):=\me(f(s+t,X_{s+t})|X_s=x)=\me(f(s+t,\alpha(t)x+X_t)),\quad(s,x)\in\widetilde{E},f\in B(\widetilde{E}),t\geq0,
\end{eqnarray*}
is a Feller semigroup. Let $\widetilde{D}$ denote the space of functions $f \in\widehat{C}(\widetilde{E})$ of the form $f(s,x)=\sum_{i=1}^{l}g_i(s)h_i(x)$ with $l\in\nz,h_i\in D$ and $g_i\in C_1([0,\infty))$ such that $g_i,g_i'\in\widehat{C}([0,\infty))$ for $i=1,\ldots,l$. Proposition \ref{prop_app_gen} states that $\widetilde{D}$ is a core for the generator $\widetilde{A}$ of $(\widetilde{T}_t)_{t\geq0}$ and
\begin{eqnarray}
	\widetilde{A}f(s,x)=\frac{\partial}{\partial s}f(s,x)+A\pi f(s,x),\qquad(s,x)\in\widetilde{E},f\in\widetilde{D}.
\label{eq_limit_gen_time_space}
\end{eqnarray}
The 'semigroup' $(T_{s,t}^{(n)})_{s,t\geq0}$ of $X^{(n)}$ is given by
\begin{eqnarray*}
	T_{s,t}^{(n)}f(x)&:=&\me(f(X_{s+t}^{(n)})|X_s^{(n)}=x)=\me(f(\log N_{s+t}^{(n)}-\alpha(s+t)\log n|N_s^{(n)}=k)
	\\&=&\me(f(\log N_{t}^{(k)}-\alpha(s+t)\log n)),\qquad(s,x)\in\widetilde{E}_n, f\in B(\rz), t\geq0.
\end{eqnarray*}
The 'generator' $(A_s^{(n)})_{s\geq0}$ of $(T_{s,t}^{(n)})_{s,t\geq0}$ is given by
\begin{eqnarray}
	A_s^{(n)}f(x) 
	&:=&\lim_{t\to 0}t^{-1}(T_{s,t}^{(n)}f(x)-f(x))\nonumber\\
	&=&\lim_{t\to 0}t^{-1}(\me(f(\log N_{t}^{(k)}-\alpha(s+t)\log n))-f(x))\nonumber\\
	&=&-f'(x)\alpha'(s)\log n+\sum_{j=1}^{k-1}(f(x+\log\tfrac{j}{k})-f(x))q_{k,j},\quad(s,x)\in\widetilde{E}_n.
\label{eq_gen_n_space}
\end{eqnarray}
Here $f\in C_1(\rz)$ such that $f,f'\in\widehat{C}(\rz)$. The semigroup $(\widetilde{T}_t^{(n)})_{t\geq0}$ of $\widetilde{X}^{(n)}$, given by
\begin{eqnarray*}
	\widetilde{T}_t^{(n)}(s,x)&:=&\me(f(s+t,X_{s+t})|X_s^{(n)}=x)\\&=&\me(f(s+t,\log N_t^{(k)}-\alpha(s+t)\log n)),\qquad(s,x)\in\widetilde{E}_n,f\in B(\widetilde{E}_n),t\geq0,n\in\nz,
\end{eqnarray*}
is a Feller semigroup on $\widehat{C}(\widetilde{E}_n)$ for every $n\in\nz$. On $\widetilde{D}$, or more precisely, for the restriction of $f\in\widetilde{D}$ to $\widetilde{E}_n$ the generator $\widetilde{A}^{(n)}$ of $\widetilde{T}^{(n)}$ is given by
\begin{eqnarray}
	\widetilde{A}^{(n)}f(s,x)=\frac{\partial}{\partial s}f(s,x)+A_s^{(n)}\pi f(s,x),\qquad(s,x)\in\widetilde{E}_n,n\in\nz.
\label{eq_n_gen_time_space}
\end{eqnarray}

\subsubsection{Proof of Corollary \ref{thm_BS} }

\begin{proof} [of Corollary \ref{thm_BS}] 
	Let $f \in D$. The approach to the proof is the same as in Section \ref{sec_proof_dust}, but the function $u\mapsto f(x+\log(1-u)),$ $u\in[0,1]$, demands second order approximation like in the integral part of the limiting generator (\ref{eq_lim_process_generator}). Define $h:[0,1]\times\rz\to\rz$ via $h(u,x):=u^{-2}(f(x+\log(1-u))-f(x)+uf'(x)),$ $u\in(0,1),$ $h(0,x):=\lim_{u\searrow 0}h(u,x)=2^{-1}(f''(x)-f'(x))$ and, since $f\in\widehat{C}(\rz)$, $h(1,x):=\lim_{u\nearrow1}h(u,x)=f'(x)-f(x)$ for $x\in\rz$. Taylor's theorem applied to $u\mapsto f(x+\log(1-u)),$ $u<1$, with evaluation point $u=0$ and exact integral remainder yields
	\begin{eqnarray*}
		h(u,x)&=&u^{-2}\int_{0}^{u}\frac{u-s}{(1-s)^2}(f''(x+\log(1-s))-f'(x+\log(1-s))){\rm d}s
		\\&=&\int_{0}^{1}\frac{1-s}{(1-us)^2}(f''(x+\log(1-us))-f'(x+\log(1-us))){\rm d}s,\quad u\in[0,1),x\in\rz.
	\end{eqnarray*}
	The latter formula of $h(u,x)$ shows that $h$ is bounded even as $u$ tends to zero. Putting $k=k(s,x,n)=e^xn^{\alpha(s)}$ in (\ref{eq_gen_n_space}) yields
	\begin{eqnarray*}
		A_s^{(n)}f(x)&=&f'(x)R(k,x)+S(k,x),\qquad(s,x)\in\widetilde{E}_n,n\in\nz,
	\end{eqnarray*}
	where
	\begin{eqnarray}
		R(k,x):=\log k-\sum_{j=1}^{k-1}\tfrac{k-j}{k}q_{k,j}-x,\qquad k\in\nz, x\in\rz,
	\label{eq_proof_BS_R(k,x)}
	\end{eqnarray}
	and
	\begin{eqnarray}
		S(k,x):=\sum_{j=1}^{k-1}(f(x+\log\tfrac{j}{k})-f(x)+\tfrac{k-j}{k}f'(x))q_{k,j},\qquad k\in\nz, x\in\rz.
	\label{eq_proof_BS_sum}
	\end{eqnarray}
	Further define $I(x):=\int_{[0,1]}h(u,x)\Lambda({\rm d}u),$ $x\in\rz$. By Eq. (\ref{eq_res_block_jump_rates_beta}) with $a=b=1$, $\tfrac{k-j}{k}q_{k,j} = (k-j+1)^{-1},$ $j\in\{1,\ldots,k-1\},k\geq 2,$ such that $\sum_{j=1}^{k-1}\tfrac{k-j}{k}q_{k,j}=\sum_{j=2}^{k}j^{-1}$ for $k\geq2$. As $n\to\infty$, $k=k(s,x,n)=e^xn^{\alpha(s)}\to\infty$ or $x\to-\infty$. Fix $T>0$. E.g., if $s\in[0,T]$ then either $k\geq n^{\alpha(T+\delta)}$ or $x<-\alpha(T)(1-\alpha(\delta))\log n$, where $\delta>0$ is a constant. The well-known asymptotics of the harmonic numbers states that $\sup_{x\in\rz}|R(k,x)-(1+\Psi(1)-x)|=|\log k-\sum_{j=1}^{k}j^{-1}-\Psi(1)|\rightarrow0$ as $k\to\infty$. Clearly, $\lim_{x\to-\infty}|f'(x)|=0$. Dividing the state space as above therefore implies
	\begin{eqnarray}
		\lim\limits_{n\to\infty}\sup_{(s,x)\in\widetilde{E}_n,s\in[0,T]}|f'(x)|| R(k,x)-(1+\Psi(1)-x)|=0.	
	\label{eq_proof_BS_conv_R(k,x)}
	\end{eqnarray}
	Substituting $k-j-1$ for $j$ in (\ref{eq_proof_BS_sum}) yields
	\begin{eqnarray*}
		S(k,x)&=&\sum_{j=0}^{k-2}(f(x+\log(1-\tfrac{j+1}{k}))-f(x)+\tfrac{j+1}{k}f'(x))q_{k,k-j-1}
		\\&=&\sum_{j=0}^{k-2}h(\tfrac{j+1}{k},x)\frac{(j+1)^2}{k^2}\binom{k}{j+2}\int_{[0,1]}u^j(1-u)^{k-2-j}\Lambda({\rm d}u)
		\\&=&\frac{k-1}{k}\sum_{j=0}^{k-2}h(\tfrac{j+1}{k},x)\frac{j+1}{j+2} \binom{k-2}{j}\int_{[0,1]}u^{j}(1-u)^{k-2-j}\Lambda({\rm d}u),\qquad k\in\nz,x\in\rz.
	\end{eqnarray*}
	Set $c:=\Lambda([0,1])>0$ and define the probability measure $\mathrm{Q}$ on $([0,1],\mathcal{B}\cap[0,1])$ as $Q:=c^{-1}\Lambda$. Let the random variables $Z_k,$ $k\in\nz,$ have distribution given by
	\begin{eqnarray*}
		\pr(Z_k=j)=\binom{k-2}{j}\int_{[0,1]}u^{j}(1-u)^{k-2-j}\mathrm{Q}({\rm d}u),\qquad j\in\{0,\ldots,k-2\},
	\end{eqnarray*}
	i.e., $Z_k$ has a mixed binomial distribution with sample size $k-2$ and random success rate $\mathrm{Q}$.  Let $Z$ have distribution $\mathrm{Q}$. Then
	\begin{eqnarray*}
		S(k,x)=c(1-k^{-1})\me((1-(Z_k+2)^{-1})h((Z_k+1)/k,x)),\qquad k\in\nz, x\in\rz, 
	\end{eqnarray*}
	and $I(x)=c\me(h(Z,x)),$ $x\in\rz$. It is easy to check that $(Z_k+1)/k\rightarrow Z$ in distribution as $k\to\infty$. The family of functions $\{h(\cdot,x)|x\in\rz\}$ is equicontinuous on $[\delta,1-\delta]$ for every $0<\delta<1/2$ and uniformly bounded on $[0,1]$. Due to $Q(\{0\})=c^{-1}\Lambda(\{0\})=0$, $Z_k\to\infty$ a.s. as $k\to\infty$ and thus $\lim_{k\to\infty}\me(1/(Z_k+2))=0$ such that the additional factor $1-(Z_k+2)^{-1}$ in the mean above can be omitted when considering the limit of $S(k,x)$ as $k\to\infty$. From Lemma \ref{local_lemma} it follows that 
	\begin{eqnarray}
		\lim_{k\to\infty}\sup_{x\in\rz}|S(k,x)-I(x)|=0.
	\label{eq_proof_BS_1}
	\end{eqnarray}  
	From $\lim_{x\to-\infty}h(Z,x)=0$ a.s., the fact that the functions $h(\cdot,x),$ $x\in\rz,$ are uniformly bounded and the dominated convergence theorem it follows that
	\begin{eqnarray}
		\lim_{x\to-\infty}|I(x)|=c\lim_{x\to-\infty}|\me(h(Z,x))|=0.
	\label{eq_proof_BS_2}
	\end{eqnarray}
	Since $f,f'\in\widehat{C}(\rz)$, $\lim_{x\to-\infty}S(k,x)=0$ for any $k\in\nz$ and, in view of (\ref{eq_proof_BS_1}) and (\ref{eq_proof_BS_2}),
	\begin{eqnarray}
		\lim_{x\to-\infty}\sup_{k\in\nz}|S(k,x)|=0.
	\label{eq_proof_BS_3}
	\end{eqnarray}
	As seen in the proof of Corollary \ref{thm_Dust}, Eqs. (\ref{eq_proof_BS_1})-(\ref{eq_proof_BS_3}) imply 
	\begin{eqnarray}
		\lim_{n\to\infty}\sup_{(s,x)\in\widetilde{E}_n,s\in[0,T]}|S(k,x)-I(x)|=0.
	\label{eq_proof_BS_S_and_I}
	\end{eqnarray}
	By (\ref{eq_proof_BS_conv_R(k,x)}), $\lim_{n\to\infty}\sup_{(s,x)\in\widetilde{E}_n,s\in[0,T]}|A_s^{(n)}f(x)-Af(x)|=0.$ Due to (\ref{eq_limit_gen_time_space}) and (\ref{eq_n_gen_time_space}), $\lim_{n\to\infty}\sup_{(s,x)\in\widetilde{E}_n,s\in[0,T]}|\widetilde{A}^{(n)}f(s,x)-\widetilde{A}f(s,x)|=0$
	for every function $f$ belonging to the core $\widetilde{D}$ and each $T>0$. From \cite[IV, Corollary 8.7]{ethierkurtz} it follows that $\widetilde{X}^{(n)}\to\widetilde{X}$ in $D_{\widetilde{E}}[0,\infty)$, hence $X^{(n)}\to X$ in $D_{\rz}[0,\infty)$ as $n\to\infty$.
\hfill$\Box$\end{proof}

\begin{remark}
	\begin{enumerate}
	\item Note that if $\Lambda=\lambda$ then $Z_k$ has a discrete uniform distribution on $\{0,\ldots,k-2\}$ and $Z$ has a continuous uniform distribution on $(0,1)$.
	\item Put $\gamma(k):=\sum_{j=1}^{k-1}\tfrac{k-j}{k}q_{k,j}=\sum_{j=2}^{k}(j-1)\binom{k}{j}\lambda_{k,j}$ for $k\geq2$. Among dust-free $\Lambda$-coalescents that stay infinite the proof works for the Bolthausen--Sznitman coalescent, because the precise asymptotics of $\gamma(k)/k=\log k-1$ as $k\to\infty$ is known. Observe that in the proof of Corollary \ref{thm_BS} the fact that $\Lambda=\lambda$ is only used to verify (\ref{eq_proof_BS_conv_R(k,x)}).
	\end{enumerate}
\end{remark}

\subsection{Proof of Theorem \ref{thm_main_convergence}} \label{sec_main}

In this section $\Lambda$ satisfies Assumption A. We continue to use the time-space setting and the notation of Subsection \ref{sec_proof_BS_time_space} with $\alpha$ replaced by $\alpha:=\alpha(t):=e^{-bt},$ $t\geq0$. Define $\Lambda_D:=\Lambda-b\lambda$ and let $\Lambda_D^{+},\Lambda_D^{-}$ denote the nonnegative measures constituting the Jordan decomposition $\Lambda_D=\Lambda_D^{+}-\Lambda_D^{-}$ of $\Lambda_D$. The decomposition of $\Lambda$ into a 'Bolthausen--Sznitman part' $b\lambda$ and a 'dust part' $\Lambda_D$ is transferred to the jump rates and the generator. Proving Theorem \ref{thm_main_convergence} now only requires to suitable arrange equations already obtained in Sections \ref{sec_proof_dust} and \ref{sec_proof_BS}. To be precise, the results of Section \ref{sec_proof_dust} are applied to the summands $\Lambda_D^{\pm}$ of $\Lambda_D$, but we omit this detail in the following.

\begin{proof} [of Theorem \ref{thm_main_convergence}] 
	Let $q_{k,j}^{\lambda},q_{k,j}^{D,+}$ and $q_{k,j}^{D,-}$ denote the rates of the block counting process corresponding to $\lambda,\Lambda_D^{+}$ and $\Lambda_D^{-}$, respectively, and define $q_{k,j}^{D}:=q_{k,j}^{D,+}-q_{k,j}^{D,-}$ for $j\in\{1,\ldots,k\}$ and $k\in\nz$. Obviously, $q_{k,j}=bq_{k,j}^{\lambda}+q_{k,j}^{D}$. Recall that $k=k(s,x,n)=e^xn^{\alpha(s)} \in \nz$ for $(s,x)\in\widetilde{E}_n$ and $n\in\nz$. From (\ref{eq_gen_n_space}) it follows that the 'generator' $A_s^{(n)}$ of $X^{(n)}=(X_t^{(n)})_{t\geq0}=(\log N_t^{(n)}-\alpha(t)\log n)_{t\geq0}$ is given by
	\begin{eqnarray*}
		A_s^{(n)}f(x)=bR(k,x)f'(x)+bS_{BS}(k,x)+S_D(k,x),\qquad(s,x)\in\widetilde{E}_n,n\in\nz,
	\end{eqnarray*}
	where
	\begin{eqnarray*}
		R(k,x)&:=&\log k-\sum_{j=1}^{k-1}\tfrac{k-j}{k}q_{k,j}^{\lambda}-x,\qquad k\in\nz,x\in\rz,
		\\S_{BS}(k,x)&:=&\sum_{j=1}^{k-1}(f(x+\log\tfrac{j}{k})-f(x)+\tfrac{k-j}{k}f'(x))q_{k,j}^{\lambda},\qquad k\in\nz,x\in\rz,
		\\S_D(k,x)&:=&\sum_{j=1}^{k-1}(f(x+\log\tfrac{j}{k})-f(x))q_{k,j}^{D},\qquad k\in\nz,x\in\rz,
	\end{eqnarray*}
	are defined as in (\ref{eq_proof_BS_R(k,x)}), (\ref{eq_proof_BS_sum}) and (\ref{eq_proof_dust_sum_and_integral}),
	and $f\in C_1(\rz)$ such that $f,f'\in\widehat{C}(\rz)$. By Lemma \ref{lem_lim_process_generator} and (\ref{eq_results_a}), the generator $A$ of $X=(X_t)_{t\geq 0}$ can be written as
	\begin{eqnarray*}
		Af(x)&=&b(1+\Psi(1)-x)f'(x)+b\int_{[0,1]}\frac{f(x+\log(1-u))-f(x)+uf'(x)}{u^2}\lambda({\rm d}u)\\&&+ \int_{[0,1]}\frac{f(x+\log(1-u))-f(x)}{u^2}\Lambda_D({\rm d}u),\qquad x\in\rz,f\in D.
	\end{eqnarray*}
	From Eqs. (\ref{eq_proof_BS_conv_R(k,x)}), (\ref{eq_proof_BS_S_and_I}) and (\ref{eq_proof_dust_1})-(\ref{eq_proof_dust_3}) it follows that $\lim_{n\to\infty}\sup_{(s,x)\in\widetilde{E}_n,s\in[0,T]}|A_s^{(n)}f(x)-Af(x)|=0$ for $f\in D$. Due to (\ref{eq_limit_gen_time_space}) and (\ref{eq_n_gen_time_space}),
	\begin{eqnarray}
		\lim_{n\to\infty}\sup_{(s,x)\in\widetilde{E}_n,s\in[0,T]}|\widetilde{A}^{(n)}f(s,x)-\widetilde{A}f(s,x)|=0
	\end{eqnarray}
	for every $f\in \widetilde{D}$ and $T>0$. By Proposition \ref{prop_app_gen}, the space $\widetilde{D}$ is a core for $\widetilde{A}$. Thus, it follows from  \cite[IV, Corollary 8.7]{ethierkurtz} that $\widetilde{X}^{(n)}\to\widetilde{X}$ in $D_{\widetilde{E}}[0,\infty)$, hence $X^{(n)}\to X$ in $D_{\rz}[0,\infty)$ as $n\to\infty$.
\hfill$\Box$\end{proof}

\subsection{Proof of Theorem \ref{thm_res_fix_line}} \label{sec_proof_fix_line}

In this section $\Lambda$ satisfies Assumption A. The process $Y^{(n)}=(Y_t^{(n)})_{t\geq0}=(\log L_t^{(n)}-e^{bt}\log n)_{t\geq0}$ is a possibly time-inhomogeneous Markov process, hence we set up the time-space framework. We provide two proofs. Using Theorem \ref{thm_main_convergence} and Siegmund-duality, in the first proof the convergence of the one-dimensional distributions and subsequently the uniform convergence of the semigroups is shown. The second proof, in which the uniform convergence of generators is shown, resembles previous ones.

\begin{proof} [First proof of Theorem \ref{thm_res_fix_line}]
	For $x\in\rz$ and $t\geq0$ define $m:=\lceil e^{y}n^{e^{bt}}\rceil\in\nz$. If $\varrho_t((-\infty,0))=\int_{[0,1]}u^{-2}\Lambda({\rm d}u)=\infty$ then $X_t$ has a continuous distribution for every $t>0$. 
	Eq. (\ref{eq_res_dual_discrete}) and Theorem \ref{thm_main_convergence} imply that
	\begin{eqnarray}
		\pr(Y_t^{(n)}\geq y)&=&\pr(L_t^{(n)}\geq m)=\pr(N_t^{(m)}\leq n)=\pr(X_t^{(m)}\leq \log n-e^{-bt}\log m)\nonumber\\&\rightarrow&\pr(X_t\leq-e^{-bt}y)=\pr(-e^{bt}X_t\geq y),\qquad y\in\rz,t\geq0,
	\label{eq_Thm_4_local_5}
	\end{eqnarray}
	as $n\to\infty$. If $\int_{[0,1]}u^{-2}\Lambda({\rm d}u)<\infty$ then the dust condition is satisfied such that $b=0$ and (\ref{eq_Thm_4_local_5}) holds true for $-y$ in the set $C_{X_t}$ of continuity points of $X_t$. Since $Y_t\overset{\text{d}}{=}-e^{bt}X_t$, $\lim_{n\to\infty}\pr(-Y_t^{(n)}\leq-y)=\pr(-Y_t\leq -y)$ for every $-y\in C_{X_t}=C_{-Y_t}$. Thus, $Y_t^{(n)}$ converges in distribution to $Y_t$ as $n\to\infty$ for every $t\geq0$. 
	
	Define the time-space processes $\widetilde{Y}^{(n)}:=(t,Y_t^{(n)})_{t\geq0},$ $n\in\nz,$ and $\widetilde{Y}:=(t,Y_t)_{t\geq0}$. The processes $\widetilde{Y}^{(n)}$ and $\widetilde{Y}$ are time-homogeneous Markov processes with state spaces $\widetilde{E}_{n}=\{(s,y)|s\geq0,e^{y}n^{e^{bs}}\in\{n,n+1,\ldots\}\}$ and $\widetilde{E}=[0,\infty)\times\rz$ and semigroups $(\widetilde{T}_t^{(n)})_{t\geq0}$ and $(\widetilde{T}_t)_{t\geq0}$. Define $k:=k(s,y,n):=e^yn^{e^{bs}}\in\{n,n+1,\ldots\}$ for $(s,y)\in\widetilde{E}_n$ and $n\in\nz$. Then
	\begin{eqnarray*}
		\widetilde{T}_t^{(n)}f(s,y)&=&\me(f(s+t,Y_{s+t}^{(n)})|Y_s^{(n)}=y)=\me(f(s+t,\log L_t^{(k)}-e^{b(t+s)}\log n))\\&=&\me(f(s+t,e^{bt}y+Y_t^{(k)})),\qquad (s,y)\in\widetilde{E}_{n},f\in B(\widetilde{E}),t\geq 0,n\in\nz.
	\end{eqnarray*}
	Fix $t>0$ and first let $f\in B(\widetilde{E})$ be of the form $f(s,y)=g(s)h(y),$ $(s,y)\in\widetilde{E}$, where $g\in B([0,\infty))$ and $h\in\widehat{C}(\rz)$. Clearly, $\widetilde{T}_t^{(n)}f(s,y)=g(s+t)\me(h(e^{bt}y+Y_t^{(k)})),$ $(s,y)\in \widetilde{E}_n,n\in\nz,$ and $\widetilde{T}_tf(s,y)=\me(f(s+t,Y_{s+t})|Y_s=y)=g(s+t)T_th(y)=g(s+t)\me(h(e^{bt}y+Y_t)),$ $(s,y)\in\widetilde{E},$ where the distribution of $Y_t$ is defined by its characteristic function $\chi_t$, given by (\ref{eq_res_char_funct_Y_t}). Note that $h$ is uniformly continuous and bounded. For $y\in\rz$ define the function $h_y:\rz\to\rz$ via $h_y(x):=h(e^{bt}y+x),$ $x\in\rz$. The family of functions $\{h_y|y\in\rz\}$ is equicontinuous and uniformly bounded. From the weak convergence of $Y_t^{(k)}$ to $Y_t$ as $k\to\infty$ and \cite[Theorem 3.1]{rangarao} it follows that $\lim_{k\to\infty}\sup_{y\in\rz}|\me(h(e^{bt}y+Y_t^{(k)}))-\me(h(e^{bt}y+Y_t))|=0.$ Since $k=e^yn^{e^{bs}}\geq n$ for $(s,y)\in\widetilde{E}_n$ and $n\in\nz$, $\lim_{n\to\infty}\sup_{(s,y)\in\widetilde{E}_n}|\me(h(e^{bt}y+Y_t^{(k)}))-\me(h(e^{bt}y+Y_t))|=0$. Thus,
	\begin{eqnarray}
		\lim_{n\to\infty}\sup_{(s,y)\in\widetilde{E}_n}|\widetilde{T}_t^{(n)}f(s,y)-\widetilde{T}_tf(s,y)|=0.
	\label{eq_Thm_4_local_4}	
	\end{eqnarray} 
	The algebra of functions $f\in B(\widetilde{E})$ of the form $f(s,y)=\sum_{i=1}^{l}g_i(s)h_i(y),$ $(s,y)\in\widetilde{E},$ where $l\in\nz,g_i\in B([0,\infty))$ and $h_i\in\widehat{C}(\rz)$, separates points and vanishes nowhere. According to the Stone--Weierstrass theorem for locally compact spaces (see e.g. \cite{debranges59}) it is a dense subset of $B(\widetilde{E})$ such that (\ref{eq_Thm_4_local_4}) holds true for $f\in B(\widetilde{E})$. \cite[IV, Theorem 2.11]{ethierkurtz} states that $\widetilde{Y}^{(n)}\to \widetilde{Y}$ in $D_{\widetilde{E}}[0,\infty)$, hence $Y^{(n)}\to Y$ in $D_{\rz}[0,\infty)$ as $n\to\infty$.
\hfill$\Box$\end{proof} 
	
The process $Y$ defined by (\ref{eq_res_semigroup_lim_fix_line}) and (\ref{eq_res_char_funct_Y_t}) is an Ornstein--Uhlenbeck type process (with nonnegative linear drift) as in \cite{satoyamazato84}. The underlying infinitely divisible distribution has characteristic exponent $y\mapsto\psi(-y),$ $y\in\rz.$ According to \cite[Theorem 3.1]{satoyamazato84}, $D$ is a core for the corresponding generator $A$ and
\begin{eqnarray}
	Af(y)=f'(y)(-a+by)+\int_{[0,1]}(f(y-\log(1-u))-f(y)-uf'(y))u^{-2}\Lambda({\rm d}u)
\label{eq_Thm_4_generator}
\end{eqnarray}
for $y\in\rz$ and $f\in D$; comparatively see Lemma \ref{lem_lim_process_generator} and its proof. 

\begin{proof} [Second proof of Theorem \ref{thm_res_fix_line}]
	The 'generator' $(A_s^{(n)})_{s\geq0}$ of $Y^{(n)}$ is given by
	\begin{eqnarray*}
		A_s^{(n)}f(y)=-f'(y)be^{bs}\log n+\sum_{j>e^yn^{e^{bs}}}(f(\log j-e^{bs}\log n)-f(y))\gamma_{e^yn^{e^{bs}},j},\quad (s,y)\in\widetilde{E}_n,n\in\nz.
	\end{eqnarray*}
	Here $f\in C_1(\rz)$ such that $f,f'\in\widehat{C}(\rz)$. Putting $k:=k(s,y,n):=e^yn^{e^{bs}}$ for $(s,y)\in\widetilde{E}_n$ and $n\in\nz$ yields
	\begin{eqnarray*}
		A_s^{(n)}f(y)=bf'(y)(-\log k+y)+\sum_{j=1}^{\infty}(f(y+\log(1+\tfrac{j}{k}))-f(y))\gamma_{k,k+j},\quad (s,y)\in\widetilde{E}_n,n\in\nz.
	\end{eqnarray*}
	Define $\Lambda_D:=\Lambda-b\lambda$ and let $\Lambda_D^{+},\Lambda_D^{-}$ denote the nonnegative measures constituting the Jordan decomposition $\Lambda_D=\Lambda_D^{+}-\Lambda_D^{-}$ of $\Lambda_D$. Let $\gamma_{k,j}^{\lambda},\gamma_{k,j}^{D,+}$ and $\gamma_{k,j}^{D,-}$ denote the jump rates of the fixation line corresponding to $\lambda,\Lambda_D^{+}$ and $\Lambda_D^{-}$, respectively, and define $\gamma_{k,j}^{D}:=\gamma_{k,j}^{D,+}-\gamma_{k,j}^{D,-}$ for $j\in\{k,k+1,\ldots\}$ and $k\in\nz$. Then $\gamma_{k,k+j}=b\gamma_{k,k+j}^{\lambda}+\gamma_{k,k+j}^{D},$ $k\in\nz,j\in\nz_0,$ and
	\begin{eqnarray}
		A_s^{(n)}f(y)=bf'(y)R(k,y)+bS_{BS}(k,y)+S_D(k,y),\qquad(s,y)\in\widetilde{E}_n,n\in\nz,
	\label{eq_proof_Thm_4_gen_discrete}
	\end{eqnarray}
	where 
	\begin{eqnarray*}
		R(k,y)&:=&-\log k+y+\sum_{j=1}^{k}\tfrac{j}{k}\gamma_{k,k+j}^{\lambda},\qquad k\in\nz,y\in\rz,\\
		S_{BS}(k,y)&:=&\sum_{j=1}^{\infty}(f(y+\log(1+\tfrac{j}{k}))-f(y)-\tfrac{j}{k}1_{[0,1]}(\tfrac{j}{k})f'(y))\gamma_{k,k+j}^{\lambda},\qquad k\in\nz,y\in\rz,\\
		S_D(k,y)&:=&\sum_{j=1}^{\infty}(f(y+\log(1+\tfrac{j}{k}))-f(y))\gamma_{k,k+j}^D,\qquad k\in\nz,y\in\rz,
	\end{eqnarray*}
	and $f\in C_1(\rz)$ such that $f,f'\in\widehat{C}(\rz)$. Using the decomposition of $\Lambda$ on Eq. (\ref{eq_Thm_4_generator}) yields
	\begin{eqnarray}
		Af(y)=bf'(y)(-1-\Psi(1)+y)+bI_{BS}(y)+I_D(y),\qquad y\in\rz,f\in D,
	\end{eqnarray}
	where
	\begin{eqnarray*}
		I_{BS}(y)&:=&\int_{[0,1]}(f(y-\log(1-u))-f(y)-uf'(y))u^{-2}\lambda({\rm d}u),\qquad y\in\rz,\\
		I_D(y)&:=&\int_{[0,1]}(f(y-\log(1-u))-f(y))u^{-2}\Lambda_D({\rm d}u),\qquad y\in\rz.
	\end{eqnarray*}
	Let $f\in D$. In the Bolthausen--Sznitman coalescent $\gamma_{k,k+j}^{\lambda}=k/(j(j+1))$ for $k,j\in\nz$ such that $\sum_{j=1}^{k}\tfrac{j}{k}\gamma_{k,k+j}^{\lambda}=\sum_{j=1}^{k}(j+1)^{-1}=H_{k+1}-1=\log k-1-\Psi(1)+o(1)$ as $k\to\infty$. Here $H_k$ denotes the $k$-th harmonic number for $k\in\nz$. Thus,
 	\begin{eqnarray}
 		\lim_{k\to\infty}\sup_{y\in\rz}|R(k,y)-(-1-\Psi(1)+y)|=0.
 	\label{eq_proof_Thm_4_R_conv}
 	\end{eqnarray}

	The function $h_{BS}:[0,1]\times\rz\to\rz$, defined via $h_{BS}(u,y):=u^{-2}(f(y-\log(1-u))-f(y)-\tfrac{u}{1-u}1_{[0,1/2]}(u)f'(y)),$ $u\in[0,1],y\in\rz,$ is bounded. 
 	Let the random variables $Z_k,$ $k\in\nz,$ have distribution given by
 	\begin{eqnarray*}
	 	\pr(Z_k=j)=\binom{k+j-2}{j-1}\int_{[0,1]}u^{j-1}(1-u)^k\lambda({\rm d}u),\qquad j,k\in\nz,
 	\end{eqnarray*}
 	i.e., $Z_k-1$ has a mixed negative binomial distribution. Observe that $h_{BS}(1-(1+\tfrac{j}{k})^{-1},y)=(\tfrac{j}{k+j})^{-2}(f(y+\log(1+\tfrac{j}{k}))-f(y)-\tfrac{j}{k}1_{[0,1]}(\tfrac{j}{k})f'(y)),$ $y\in\rz,$ 
 	and $\gamma_{k,k+j}^{\lambda}=(\tfrac{j}{k+j})^{-2}(1+(k+j)^{-1})(1-(j+1)^{-1})\pr(Z_k=j)$ for $j,k\in\nz$. Hence,
 	\begin{eqnarray*}
 		S_{BS}(k,y)=\me(h_{BS}(1-(1+Z_k/k)^{-1},y)(1+(k+Z_k)^{-1})(1-(Z_k+1)^{-1})).
 	\end{eqnarray*}
 	Let $Z$ have uniform distribution on $(0,1)$ such that $I_{BS}(y)=\me(h_{BS}(Z,y))$ for $y\in\rz$. Here it is used that $\int_{[0,1]}u^{-2}(u-\tfrac{u}{1-u}1_{[0,1/2]}(u))\lambda({\rm d}u)=\int_{0}^{1/2}-(1-u)^{-1}\lambda({\rm d}u)+\int_{1/2}^{1}u^{-1}\lambda({\rm d}u)=0.$ The function $g:(0,\infty)\to(0,1)$, defined via $g(u):=1-(1+u)^{-1},$ $u\in(0,\infty),$ is bounded and continuous. Since $Z_k/k\to Z/(1-Z)$ in distribution as $k\to\infty$, $1-(1+Z_k/k)^{-1}=g(Z_k/k)\to g(Z/(1-Z))=Z$ in distribution as $k\to\infty$. In particular, the random variables have values in $[0,1]$. When considering the limit $k\to\infty$, the factor $(1+(k+Z_k)^{-1})(1-(Z_k+1)^{-1})$ has no influence on $S_{BS}(k,y)$. From Lemma \ref{local_lemma} it follows that 
 	\begin{eqnarray}
 		\lim_{k\to\infty}\sup_{y\in\rz}|S_{BS}(k,y)-I_{BS}(y)|=0.
 	\label{eq_proof_Thm_4_conv_S_BS}
 	\end{eqnarray}
 	
	The measure $\Lambda_D$ is real-valued. Eq. (\ref{eq_proof_Thm_4_S_D_conv}) below can be proved when $\Lambda_D$ is replaced by $\Lambda_D^{+}$ and $\Lambda_D^{-}$ in this paragraph, and then holds true for $\Lambda_D$ by linearity. The function $h_D:[0,1]\times\rz\to\rz,$ defined via $h_D(u,y):=u^{-1}(f(y-\log(1-u))-f(y)),$ $u\in[0,1],y\in\rz,$ is bounded. By assumption $c:=\int_{[0,1]}u^{-1}\Lambda_D({\rm d}u)<\infty.$ Define the probability measure $Q$ on $([0,1],\mathcal{B}\cap[0,1])$ via $Q(A):=c^{-1}\int_A u^{-1}\Lambda_D({\rm d}u),$ $A\in\mathcal{B}\cap[0,1].$ Let the random variables $Z_k,$ $k\in\nz,$ have distribution given by
 	\begin{eqnarray*}
 		\pr(Z_k=j)=\binom{k+j-1}{j}\int_{[0,1]}u^{j}(1-u)^kQ({\rm d}u),\qquad j\in\nz_0,k\in\nz,
 	\end{eqnarray*}
 	i.e., $Z_k$ has a mixed negative binomial distribution. Observe that $h_D(1-(1+\tfrac{j}{k})^{-1},y)=(f(y+\log(1+\tfrac{j}{k}))-f(y))\tfrac{k+j}{j},$ $y\in\rz,$ and $\gamma_{k,k+j}^{D}=c\tfrac{k+j}{j}(1-(1+j)^{-1})\pr(Z_k=j)$ for $j,k\in\nz$. Hence,
 	\begin{eqnarray*}
 		S_D(k,y)&=&\sum_{j=0}^{\infty}(f(y+\log(1+\tfrac{j}{k})-f(y))\gamma_{k,k+j}^{D}\\&=&c\me(h_D(1-(1+Z_k/k)^{-1},y)(1-(1+Z_k)^{-1})),\qquad k\in\nz,y\in\rz.
 	\end{eqnarray*}
 	Let the random variable $Z$ have distribution $Q$. In particular, $I_D(y)=c\me(h_D(Z,y)),$ $y\in\rz$. By Lemma \ref{local_lemma} and since $1-(1+Z_k/k)^{-1}$ converges in distribution to $Z$ as $k\to\infty$, $\lim_{k\to\infty}\sup_{y\in\rz}|\me(h_D(1-(1+Z_k/k)^{-1},y))-\me(h_D(Z,y))|=0.$ Thus,
 	\begin{eqnarray}
 		\lim_{k\to\infty}\sup_{y\in\rz}|S_D(k,y)-I_D(y)|=0.
 	\label{eq_proof_Thm_4_S_D_conv}
 	\end{eqnarray}
	
	Taking into account that $k=e^yn^{e^{bs}}\geq n$ for $(s,y)\in\widetilde{E}_n$ and $n\in\nz$, Eqs. (\ref{eq_proof_Thm_4_gen_discrete})-(\ref{eq_proof_Thm_4_S_D_conv}) imply
 	\begin{eqnarray*}
 		\lim_{n\to\infty}\sup_{(s,y)\in\widetilde{E}_n}|A_s^{(n)}f(y)-Af(y)|=0.
 	\end{eqnarray*}
	The time-space variant of \cite[IV, Corollary 8.7]{ethierkurtz} as implemented in the proof of Theorem \ref{thm_main_convergence} yields the desired convergence of $Y^{(n)}\to Y$ in $D_{\rz}[0,\infty)$ as $n\to\infty$.
\hfill$\Box$\end{proof}

\subsection{Appendix}

\begin{lemma}
	The $\Lambda$-coalescent does not come down from infinity under Assumption A.
\label{lem_app_not_cdi}
\end{lemma}

\begin{proof}
	Suppose that $\Lambda$ satisfies Assumption A. Define $\Lambda_D:=\Lambda-b\lambda$, let $\Lambda_D^{+}$ and $\Lambda_D^{-}$ denote the nonnegative measures constituting the Jordan decomposition $\Lambda_D=\Lambda_D^{+}-\Lambda_D^{-}$ of $\Lambda_D$ and let $|\Lambda_D|:=\Lambda_D^{+}+\Lambda_D^{-}$ denote the (total) variation of $\Lambda_D$. Define $\eta_k^{\Lambda}:=k\sum_{j=0}^{k-2}\int_{[0,1]}(1-u)^j\Lambda({\rm d}u)$ and $\eta_k^{b\lambda}$ and $\eta_k^{|\Lambda_D|}$ similarly with $b\lambda$ and $|\Lambda_D|$ in place of $\Lambda$ for $k\geq2$. By assumption,
	\begin{eqnarray*}
		\lim_{k\to\infty}k^{-1}\eta_k^{|\Lambda_D|}=\int_{[0,1]}u^{-1}|\Lambda_D|({\rm d}u)<\infty.
	\end{eqnarray*}
	Since
	\begin{eqnarray*}
		(k\log k)^{-1}\eta_k^{b\lambda}= b(\log k)^{-1}\sum_{j=0}^{k-2}\int_0^{1}(1-u)^j{\rm d}u=b(\log k)^{-1}\sum_{j=0}^{k-2}(j+1)^{-1}\rightarrow b,\qquad k\to\infty,
	\end{eqnarray*}
	it follows that $\eta_k^{b\lambda}+\eta_k^{|\Lambda_D|}\sim bk\log k$ as $k\to\infty$. Due to $\Lambda\leq b\lambda+|\Lambda_D|$, it holds that $\eta_k^{\Lambda}\leq\eta_k^{b\lambda}+\eta_k^{|\Lambda_D|}$ for $k\geq2$. Hence,
	\begin{eqnarray*}
		\sum_{k=2}^{\infty}\big(\eta_k^{\Lambda}\big)^{-1}\geq\sum_{k=2}^{\infty}\big(\eta_k^{b\lambda}+\eta_k^{|\Lambda_D|}\big)^{-1}
		=\infty.
	\end{eqnarray*}
	The claim then follows from Schweinsberg's criterion \cite[Corollary 2]{schweinsberg00}.
\hfill$\Box$\end{proof}

\noindent The following lemma is a generalization of the integral criterion of convergence in distribution and is used in Sections \ref{sec_proof_dust}-\ref{sec_proof_fix_line} to prove the uniform convergence of generators.

\begin{lemma}
	Let $X,X_1,X_2,\ldots$ be random variables on a probability space $(\Omega,\mathcal{F},\pr)$ with values in $[0,1]$ such that $\pr(X=0)=\pr(X=1)=0$ and $X_n\to X$ in distribution as $n\to\infty$. Suppose that the family $F$ of functions $f:[0,1]\to\rz$ is uniformly bounded on $[0,1]$ and equicontinuous on $[\delta,1-\delta]$ for every $0<\delta<1/2$. In particular, $f\in F$ is bounded and continuous on $(0,1)$. Then
	\begin{eqnarray*}
		\lim_{n\to\infty}\sup_{f\in F}|\me(f(X_n))-\me(f(X))|=0.
	\end{eqnarray*}
	\label{local_lemma}
\end{lemma}

\begin{proof}
	Define $M:=\sup_{f\in F}\|f\|<\infty$ and let $\varepsilon>0$ be arbitrary. The assumption $\pr(X=0)=\pr(X=1)=0$ and the convergence of $X_n$ to $X$ in distribution as $n\to\infty$ provide the existence of $0<\delta<1/2$ and $n_0\in\nz$ such that $\pr(X_n\not\in[\delta,1-\delta])<\varepsilon/(4M)$ for $n\geq n_0$ and $\pr(X\not\in[\delta,1-\delta])<\varepsilon/(4M)$. For $f\in F$ define $\tilde{f}:[0,1]\to\rz$ via $\tilde{f}(u):=f(\delta),$ $0\leq u\leq\delta,$ $\tilde{f}(u):=f(u),$ $\delta\leq u\leq 1-\delta,$ and $\tilde{f}(u):=f(1-\delta),$ $1-\delta\leq u\leq 1$. Then $\{\tilde{f}|f\in F\}$ is bounded (by $M$) and equicontinuous on $[0,1]$. \cite[Theorem 3.1]{rangarao} yields
	\begin{eqnarray*}
		\lim_{n\to\infty}\sup_{f\in F}|\me(\tilde{f}(X_n))-\me(\tilde{f}(X))|=0.
	\end{eqnarray*}
	From
	\begin{eqnarray*}
		\vert\me(f(X_n))-\me(f(X))\vert&\leq&\me(\vert f(X_n)-\tilde{f}(X_n)\vert)\\&&+\vert\me(\tilde{f}(X_n))-\me(\tilde{f}(X))\vert+\me(\vert\tilde{f}(X)-f(X)\vert)
		\\&\leq&2M\pr(X_n\not\in[\delta,1-\delta])\\&&+2M\pr(X\not\in[\delta,1-\delta])+\vert\me(\tilde{f}(X_n))-\me(\tilde{f}(X))\vert,\qquad n\in\nz,f\in F,
	\end{eqnarray*}
	it follows that $\lim_{n\to\infty}\sup_{f\in F}\vert\me(f(X_n))-\me(f(X))\vert\leq\varepsilon$.	Since $\varepsilon>0$ is arbitrary the proof is complete.
\hfill$\Box$\end{proof}

\begin{remark} In \cite[Theorem 3.1]{rangarao} the state space is more generally a separable metric space, but equicontinuity of $F$ is required to hold on the whole state space.
\end{remark}

\noindent Let $E$ be a complete separable metric space and equip $\widetilde{E}:=[0,\infty)\times E$ with the product metric. The following proposition treats the generator of time-space processes of time-homogeneous Feller processes.

\begin{proposition}
	Suppose that $(T_t)_{t\geq0}$ is a Feller semigroup on $\widehat{C}(E)$ with generator $A$ and that $D$ is a core for $A$. For $f\in\widehat{C}(\widetilde{E})$ and $s\in[0,\infty)$ let $\pi f(s,x)$ denote the function $x \mapsto f(s,x),$ $x\in E$. The semigroup $(\widetilde{T}_t)_{t\geq 0}$, defined via
	\begin{eqnarray*}
		\widetilde{T}_tf(s,x):=T_t\pi f(s+t,x),\qquad (s,x)\in\widetilde{E}, f\in B(\widetilde{E}), t\geq 0,
	\end{eqnarray*}
	is a Feller semigroup on $\widehat{C}(\widetilde{E})$. Let $\widetilde{D}$ denote the space of functions $f\in \widehat{C}(\widetilde{E})$ of the form $f(s,x)=\sum_{i=1}^{l}g_i(s)h_i(x),$ $(s,x)\in\widetilde{E}$, where $l\in\nz, h_i\in D$ and $g_i\in C_1([0,\infty))$ such that $g_i, g_i'\in\widehat{C}([0,\infty))$ for $i=1,\ldots,l$. Then $\widetilde{D}$ is a core for the generator $\widetilde{A}$ of $(\widetilde{T}_t)_{t\geq0}$ and
	\begin{eqnarray}
		\widetilde{A}f(s,x)=\frac{\partial}{\partial s}f(s,x)+A\pi f(s,x),\qquad (s,x)\in\widetilde{E}, f\in\widetilde{D}.
	\label{eq_app_prop_gen}
	\end{eqnarray}
\label{prop_app_gen}
\end{proposition}

\begin{proof} 
	Observe that all functions involved in the proof are bounded and uniformly continuous. Clearly, the right-hand side of (\ref{eq_app_prop_gen}) lies in $\widehat{C}(\widetilde{E})$. The core $D$ is a dense subset of $\widehat{C}(E)$. Hence $\widetilde{D}$ is a dense subset of the space $D_0$ of functions $f\in\widehat{C}(\widetilde{E})$ of the form $f(s,x)=\sum_{i=1}^{l}g_i(s)h_i(x),$ $(s,x)\in\widetilde{E},$ where $l\in\nz,h_i\in\widehat{C}(E)$ and $g_i\in\widehat{C}([0,\infty))$ for $i=1,\ldots,l$. The algebra $D_0$ separates points and vanishes nowhere. The Stone--Weierstrass theorem for locally compact spaces (e.g. \cite{debranges59}) ensures that $D_0$ is a dense subset of $\widehat{C}(\widetilde{E})$. In \cite{debranges59} the theorem is stated for complex-valued functions, but it remains true for real-valued functions. To see this, let $f\in\widehat{C}(E)\subseteq\widehat{C}(E,\cz)$ be arbitrary. By the theorem there exist a sequence $(k_n)_{n\in\nz}\subseteq\widehat{C}(E,\cz)$ such that $\lim_{n\to\infty}||k_n-f\|=0$. Then $f_n:={\rm Re}(k_n)\in\widehat{C}(E),$ $n\in\nz,$ and $\|f_n-f\|\leq\|k_n-f\|\to0$ as $n\to\infty$. Thus $\widetilde{D}$ is a dense subset of $\widehat{C}(\widetilde{E})$ as well. If $h\in D$ and $g\in C_1([0,\infty))$ such that $g,g'\in\widehat{C}([0,\infty))$, then
	\begin{eqnarray*}
		t^{-1}(\widetilde{T}_tg(s)h(x)-g(s)h(x))=t^{-1}(g(s+t)-g(s))h(x)+g(s+t)t^{-1}(T_th(x)-h(x))
	\end{eqnarray*}
	converges uniformly in $(s,x)\in\widetilde{E}$ to $g'(s)h(x)+g(s)Ah(x)$ as $t\searrow0$, thus $\widetilde{D}$ lies in the domain of $\widetilde{A}$ and (\ref{eq_app_prop_gen}) holds true. By the same argument as above, the space $D_1$ of functions $f\in\widehat{C}(\widetilde{E})$ of the form $f(s,x)=\sum_{i=1}^{l}g_i(s)h_i(x),$ $(s,x)\in\widetilde{E}$, where $g_i(s)=c_i\exp(-a_is),$ $s\in[0,\infty)$ with $c_i\in\rz$ and $a_i>0$ and $h_i\in D$ for $i=1,\ldots,l$, is a dense subset of $\widehat{C}(\widetilde{E})$. By Hille--Yosida theory (see e.g. \cite[I, Proposition 3.1]{ethierkurtz}) it now suffices to show that the image of $\lambda I-\widetilde{A}\vert_{\widetilde{D}}$ is a dense subspace of $\widehat{C}(\widetilde{E})$ for some $\lambda>0$ in order to prove that $\widetilde{D}$ is a core for $\widetilde{A}$. Here $I$ denotes the identity map on $\widehat{C}(E)$ or $\widehat{C}(\widetilde{E})$. Let $\varepsilon>0$ and $f\in\widehat{C}(\widetilde{E})$ be arbitrary. By density of $D_1$ in $\widehat{C}(\widetilde{E})$, there exists $f_1\in D_1$ of the form $f_1(s,x)=\sum_{i=1}^{l}g_{i}(s)h_{i}(x),$ $(s,x)\in\widetilde{E}$, such that $\|f_1-f\|<\varepsilon/2$. Since $D$ is a core for $A$, the image of $\lambda I -A\vert_D$ is a dense subset of $\widehat{C}(E)$ for every $\lambda>0$, in particular for $\lambda+a_i$ in place of $\lambda$. Hence there exists $r_i\in D$ such that $\|(\lambda+a_i)r_i - Ar_i - h_i\|<\varepsilon/(2l\|g_i\|)$ for $i=1,\ldots,l$. Clearly, the function $(s,x)\mapsto\sum_{i=1}^{l}g_i(s)r_i(x),$ $(s,x)\in\widetilde{E}$, belongs to $\widetilde{D}$ and, by (\ref{eq_app_prop_gen}),
	\begin{eqnarray*}
		\|(\lambda I-\widetilde{A})\sum_{i=1}^{l}g_i(s)r_i(x)-f(s,x)\|
		&\leq&\|(\lambda I-\widetilde{A})\sum_{i=1}^{l}g_i(s)r_i(x)-\sum_{i=1}^{l}g_i(s)h_i(x)\|+\|f_1-f\|
		\\&\leq&\sum_{i=1}^{l}\|g_i((\lambda+a_i)r_i-Ar_i-h_i)\|+\varepsilon/2 \leq\varepsilon.
	\end{eqnarray*}
	In the second last step it is used that $g_i'(s)=-a_ig_i(s),$ $s\in[0,\infty)$ for $i=1,\ldots,l$. Since $\varepsilon>0$ has been arbitrary the proof is complete.
\hfill$\Box$\end{proof}

\begin{remark} 
	The last part of the proof of Proposition \ref{prop_app_gen} can be simplified under the additional assumption that $T_tD\subseteq D$ for every $t>0$. Then $\widetilde{T}_t \widetilde{D} \subseteq \widetilde{D}$ for every $t \geq 0$ and the claim follows by applying the core theorem \cite[I, Proposition 3.3]{ethierkurtz}.
\end{remark} 

\noindent The computations in the proof of the following proposition are based on Gau\ss' representation \cite[p. 247]{whittakerwatson96} for the digamma function 
\begin{eqnarray*}
	\Psi(z)=\int_{0}^{\infty}\bigg(\frac{e^{-u}}{u}-\frac{e^{-zu}}{1-e^{-u}}\bigg){\rm d}u,\qquad{\rm Re}(z)>0.
\end{eqnarray*}

\begin{proposition}
	Suppose that $\Lambda=\beta(1,b)$ with $b>0$. Then the measure $\varrho$, defined by (\ref{eq_res_varrho}), has density $f$ with respect to Lebesgue measure on $\rz\setminus\{0\}$ given by $f(u):=be^{bu}(1-e^{u})^{-2}$ for $u<0$ and $f(u):=0$ for $u>0$. Let $a$ and $\psi$ be given by (\ref{eq_results_a}) and (\ref{eq_results_Psi}). Then
	\begin{eqnarray}
		a=b(1+\Psi(b))
	\label{eq_app_local}
	\end{eqnarray}
	and 
	\begin{eqnarray*}
		\psi(x)=b((1-b)\Psi(b)-(1-b-ix)\Psi(b+ix)),\qquad x\in\rz.
	\end{eqnarray*}
	\label{prop_app_underlying_char_funct_BS}
\end{proposition}

\begin{proof}
	It is easily verified that $\varrho$ has density as stated in the proposition. Eq. (\ref{eq_app_local}) follows from
	\begin{eqnarray*}
		\int_{[0,1]}u^{-1}(\Lambda-b\lambda){\rm d}u&=&b\int_{0}^{1}u^{-1}((1-u)^{b-1}-1){\rm d}u\\&=&b\int_{0}^{\infty}\Big(\frac{e^{-bu}}{1-e^{-u}}-\frac{e^{-u}}{1-e^{-u}}\Big){\rm d}u=b(\Psi(1)-\Psi(b)).
	\end{eqnarray*}
	Next, note that
	\begin{eqnarray*}
		\Psi(b)-\Psi(b+ix)=\int_{0}^{\infty}(e^{-ixu}-1)\frac{e^{-bu}}{1-e^{-u}}{\rm d}u,\qquad x\in\rz,
	\end{eqnarray*}
	is the characteristic exponent of the negative of a drift-free subordinator, whose L\'{e}vy measure has density $u\mapsto e^{bu}(1-e^{u})^{-1},$ $u<0,$ with respect to Lebesgue measure on $(-\infty,0)$. If $b<1$ and $Z$ has characteristic function $\exp((1-b)(\Psi(b)-\Psi(b+ix)),$ $x\in\rz,$ then $\me(\log(1+|Z|))<\infty$. This fact is required in order to use \cite[V, Theorem 6.7]{steutelvanharn04} in Example \ref{ex_res}. Integration by parts yields
	\begin{eqnarray*}
		ix(\Psi(b+ix)-\Psi(b))&=&\int_{0}^{\infty}(ix-ixe^{-ixu})\frac{e^{-bu}}{1-e^{-u}}{\rm d}u 
		\\&=&(ixu+e^{-ixu}-1)\frac{e^{-bu}}{1-e^{-u}}\bigg\vert_{u=0}^{u=\infty}\\&&~~-\int_{0}^{\infty}(ixu+e^{-ixu}-1)\bigg( \frac{-be^{-bu}}{1-e^{-u}}-\frac{e^{-bu}}{(1-e^{-u})^2}e^{-u}\bigg){\rm d}u
		\\&=& \int_{0}^{\infty}(e^{-ixu}-1+ixu)\frac{e^{-bu}}{(1-e^{-u})^2}(1-(1-b)(1-e^{-u})){\rm d}u,\qquad x\in\rz.
	\end{eqnarray*}
	Hence,
	\begin{eqnarray*}
		&&(1-b)\Psi(b)-(1-b-ix)\Psi(b+ix)
		\\&&=ix\Psi(b)+(1-b)(\Psi(b)-\Psi(b+ix))+ix(\Psi(b+ix)-\Psi(b))
		\\&&=ix\Psi(b)+(1-b)\int_{0}^{\infty}(e^{-ixu}-1)\frac{e^{-bu}}{1-e^{-u}}{\rm d}u\\&&~~~~~+\int_{0}^{\infty}(e^{-ixu}-1+ixu)\frac{e^{-bu}}{(1-e^{-u})^2}(1-(1-b)(1-e^{-u})){\rm d}u
		\\&&=ix\Psi(b)+\int_{0}^{\infty}(e^{-ixu}-1+ixu)\frac{e^{-bu}}{(1-e^{-u})^2}{\rm d}u-ix(1-b)\int_{0}^{\infty}u\frac{e^{-bu}}{(1-e^{-u})}{\rm d}u
		\\&&=ix(\Psi(b)-(1-b)\Psi'(b))+b^{-1}\int_{\rz\setminus\{0\}}(e^{ixu}-1-ixu)\varrho({\rm d}u)
		\\&&=ix\Big(\Psi(b)-(1-b)\Psi'(b)+b^{-1}\int_{\rz\setminus\{0\}}(e^{u}-1-u)\varrho({\rm d}u)\Big)\\&&~~~~~+b^{-1}\int_{\rz\setminus\{0\}}(e^{ixu}-1+ix(1-e^{u}))\varrho({\rm d}u).
	\end{eqnarray*}
	The calculation
	\begin{eqnarray*}
		&&-(1-b)\Psi'(b)+b^{-1}\int_{\rz\setminus\{0\}}(e^{u}-1-u)\varrho({\rm d}u)\\&&~~~=\int_{0}^{\infty}(-(1-b)u(1-e^{-u})+e^{-u}-1+u)\frac{e^{-bu}}{(1-e^{-u})^2}{\rm d}u=-\frac{e^{-bu}}{1-e^{-u}}u\Big\vert_{u=0}^{u=\infty}=1
	\end{eqnarray*}
	and multiplication with $b$ complete the proof.
\hfill$\Box$\end{proof}

\end{document}